\documentclass[11pt,a4paper,twoside]{article}
\author{Stephen Bedford}
\title{Function spaces for liquid crystals}

\usepackage{amsmath,amsfonts,euscript,amssymb,amsthm,times,graphicx,bm,bbm,fancyhdr,color}
\usepackage{txfonts,epsfig,subfig,mathrsfs,enumerate,textcomp,wrapfig,float,sidecap}
\usepackage[margin=1.0in]{geometry}
\usepackage[showonlyrefs]{mathtools}
\theoremstyle{plain}
\newtheorem{theorem}{Theorem}[section]
\newtheorem{proposition}[theorem]{Proposition}
\newtheorem{definition}[theorem]{Definition}
\newtheorem{lemma}[theorem]{Lemma}
\newtheorem{corollary}[theorem]{Corollary}
\newtheorem{remark}[theorem]{Remark}

\newtheorem*{theorem*}{Theorem}
\newtheorem*{proposition*}{Proposition}
\newtheorem*{definition*}{Definition}
\newtheorem*{lemma*}{Lemma}
\newtheorem*{corollary*}{Corollary}
\newtheorem*{remark*}{Remark}

\newcommand{\tostar}{\stackrel{\ast}{\rightharpoonup}}
\newcommand{\vn}{{\bf n}}
\newcommand{\vv}{{\bf v}}
\newcommand{\vm}{{\bf m}}

\newcommand{\vu}{{\bf u}}
\newcommand{\vphi}{\boldsymbol\phi}
\newcommand{\vpsi}{\boldsymbol\psi}
\newcommand{\vQ}{{\bf Q}}

\renewenvironment{proof}{{\bf{Proof}\vspace{3 mm}}}{\qed}

\numberwithin{equation}{section}

\begin{document}

\maketitle
\mathtoolsset{showonlyrefs=true}

\begin{abstract}

We consider the relationship between three continuum liquid crystal theories: Oseen-Frank, Ericksen and Landau-de Gennes. It is known that the function space is an important part of the mathematical model and by considering various function space choices for the order parameters $s$, $\vn$, and $\vQ$, we establish connections between the variational formulations of these theories. We use these results to derive a version of the Oseen-Frank theory using special functions of bounded variation. This proposed model can describe both orientable and non-orientable defects. Finally we study a number of frustrated nematic and cholesteric liquid crystal systems and show that the model predicts the existence of point and surface discontinuities in the director.

\end{abstract}

%
%
%
%
%
%

\section{Introduction}

In the study of the calculus of variations, modelling any physical problem has two basic aspects. The Lagrangian represents the mathematical model of the free energy density of the system being studied. The second aspect is the regularity of the mappings being considered: the function space. Even though this is well known, usually the majority of the focus is given over to the exact form of the Lagrangian, even though with a fixed Lagrangian choosing a different function space can give rise to a different minimum energy for the functional. This idea was first observed by Lavrentiev \cite{lavrentiev1926sur} and has since been known as the Lavrentiev phenomenon. This has particular relevance in the study of liquid crystals because of the abundance of continuum theories and order parameters. For the purposes of this paper we shall investigate three continuum theories: Oseen-Frank, Ericksen, and Landau-de Gennes. Each of these theories uses a different order parameter. The Landau-de Gennes theory uses a symmetric and traceless matrix $\vQ$, Ericksen's theory uses a scalar order parameter $s$ in conjunction with a director $\vn$, whilst the Oseen-Frank theory just uses a director $\vn$. Each theory will be formally introduced in the next section. The Oseen-Frank theory is conceptually the simplest but it is flawed when the function space choice for the  director field $\vn$ is  $ W^{1,2}\left( \Omega,\mathbb{S}^2 \right)$. This is because it does not give finite energy to the $\pm \frac{1}{2}$ defects \cite{ball2011orientability} and it does not respect the head to tail symmetry of the molecules. However we will show in this paper that if a form of the Oseen-Frank free energy  is considered where the director field is a special function of bounded variation, these problems can be addressed.

\vspace{3 mm}

The paper itself is split into two main sections. In the first we investigate the relationship between the three outlined continuum theories. By writing the matrix order parameter $\vQ$ in the uniaxial form, where two of its eigenvalues are equal, so that
\begin{equation}\label{0.1}
 \vQ:=s\left( \vn \otimes \vn -\frac{1}{3} \right),
\end{equation}
it is clear that there will be some kind of connection between these three different theories. By considering different choices of function spaces for the variables $s,\vn$, and $\vQ$, as well as the domain topology and dimension, we will prove Theorems \ref{3:1D-Uniaxial-Ericksen}, \ref{3:1D-Biaxial-Ericksen}, \ref{3:2,3D Uniaxial Ericksen}, and \ref{3:2,3D Biaxial Ericksen}. All of these results show that the uniaxial or biaxial Landau-de Gennes model is equivalent to a specific Ericksen model under suitable function space assumptions. Importantly we prove that if $\Omega \subset \mathbb{R}^d$ where $d>1$ then the two theories are only equivalent if we utilise director fields which are special functions of bounded variation (SBV). Functions of bounded variation have previously been suggested for the study of liquid crystals (see \cite[Section 4.6]{ambrosio2000functions} or \cite{aviles1987mathematical}) and Theorems \ref{3:2,3D Uniaxial Ericksen} and \ref{3:2,3D Biaxial Ericksen} show that this suggestion was warranted.

\vspace{3mm}

The main technical issue which we need to consider for our equivalence results is finding a unit vector field $\vm$ which corresponds to a given line field $\vn\otimes\vn$. This lifting question in Sobolev and BV spaces has been studied in its own right in a number of papers \cite{bourgain2000lifting,davila2003lifting}. However we will use our lifting results to prove the connections between the Landau-de Gennes theory and Ericksen's theory. In total, we will prove three lifting results, Propositions \ref{3:prop1}, \ref{3:prop2 Ball-Zarnescu Extension}, and \ref{3:prop3 Line field Vector field SBV}. The different results are needed because the problem can be very different depending on the ambient dimension and whether the domain is simply connected or not. Proposition \ref{3:prop3 Line field Vector field SBV} is particularly interesting as it illustrates the fact that the vector field cannot always retain the same regularity as the line field.

\vspace{3 mm}

 The basic idea that these theories are part of some hierarchy has received some attention in recent years. In papers by Majumdar, Zarnescu, and Nguyen \cite{majumdar2010landau,nguyen2013refined}, minimisers of the Landau-de Gennes free energy were shown to converge to minimisers of the Oseen-Frank free energy in the vanishing elastic constant limit. This limit essentially forces the matrix order parameter to become uniaxial in order to minimise the dominant bulk energy term. We will consider a similar limit in Section \ref{section: new approach} where we show that by considering a vanishing elastic constant limit, we can justify an Oseen-Frank model using director fields
 \begin{equation}\label{0.2}
   \vn \in SBV^2\left( \Omega ,\mathbb{S}^2 \cup \left\{ 0 \right\} \right).
 \end{equation}
 The second half of the paper is devoted to this exploration of this proposed model. We remind the reader that the basic definitions and properties of functions of bounded variation can be found in \cite{ambrosio2000functions}. Any relevant results pertaining to functions of bounded variation will be introduced in due course. 
 
 \vspace{3 mm}

When trying to model nematic or cholesteric liquid crystals, it is known that many of the commonly used molecules are rod-like. Therefore any model should respect this natural invariance. The Landau-de Gennes theory does this by utilising a matrix order parameter with terms such as $\vn \otimes \vn$ to ensure the symmetry. One of the main conclusions of the first half of the paper is that if we wish to use the director as a simpler order parameter, then we must consider director fields of bounded variation for full generality. Such a model is considered in the latter half of the paper and the basics of the problem are established. We prove the existence of a minimiser, we find various forms of the Euler Lagrange equation, we discover solutions in simple cases, and we find minimisers for specific liquid crystal problems.

\vspace{3 mm}

Special functions of bounded variation are a degree more technical than Sobolev functions, therefore it is important to show that the extra complexity yields more accurate predictions of liquid crystal behaviour. This is what we aim to achieve in Sections \ref{section: SBV}-\ref{section: cholesteric problem}, where we show that by studying models with director fields of such regularity, the predictions made using Sobolev spaces can be extended. We will study an Oseen-Frank energy using functions of bounded variation $\vn \in SBV^2 \left( \Omega, \mathbb{S}^2 \cup \left\{ 0 \right\} \right)$ on various domains. In particular we will examine the widely studied cuboid domain problem, where
\begin{equation}\label{4.6c}
 \Omega = (-L_1,L_1)\times (-L_2,L_2) \times (0,L_3),
\end{equation}
with frustrated boundary conditions on the top and bottom faces. Proposition \ref{4: cholesteric multidimensional minimisers} shows that for large values of the confinement ration $\frac{L_3}{P}$, where $P$ is the pitch of the cholesteric, the minimiser of the Oseen-Frank energy must depend on more than a single variable. Furthermore any function of just $z$ cannot even be a local minimiser for the SBV problem with a large enough confinement ratio. Both of these results have been known experimentally for some time as the existence of stable multi-dimensional cholesteric structures such as cholesteric fingers \cite{oswald2000static,smalyukh2005electric}, double twist cylinders \cite{kikuchi2002polymer,wright1989crystalline}, and torons \cite{smalyukh2010three}, has been thoroughly proven and investigated. However using our proposed model allows the problem to addressed analytically. The key to this development is using functions of bounded variation because the results proved elusive when considering director fields $\vn \in W^{1,2} \left( \Omega ,\mathbb{S}^2 \right)$ \cite{Bedford2014}.

\section{Continuum theories and Preliminaries}

The oldest continuum theory of liquid crystals was formulated by Oseen and Frank \cite{oseen1933theory,frank1958theory} in the first half of the twentieth century. In general the free energy density that they derived depended on four material parameters, but in this paper we will consider the simplified functional using the one constant approximation so that it can be written in the form
\begin{equation}\label{4.1}
 I_O(\vn):=\int_\Omega |\nabla\vn|^2+2t\vn\cdot\nabla \times \vn+t^2\,dx,
\end{equation}
where $\vn :\Omega\rightarrow \mathbb{S}^2$. The liquid crystal molecules are rod-like, and this theory was developed using a coarse-graining approach on the supposition that the energy of the system can be modelled through the elastic strains of their principal axis $\vn$. In the 1990s Ericksen proposed an extension of this theory with the addition of a scalar order parameter $s:\Omega\rightarrow \mathbb{R}$ to represent the degree of orientation of the molecules. The director represents an average local orientation, and the scalar order parameter represents the local variance from this mean. In this setting the free energy has the form \cite{virga1994variational}
\begin{equation}\label{4.2}
 I_E(\vn,s)=\int_\Omega 2s^2\left( |\nabla\vn|^2+2t\vn\cdot\nabla \times \vn+t^2\right) +\frac{2}{3}|\nabla s|^2+\sigma(s)\,dx,
\end{equation}
where $\sigma(s):=\frac{a}{2}s^2+\frac{b}{3}s^3+\frac{c}{4}s^4$. The function $\sigma$ is the bulk energy which represents the competition between the isotropic and ordered phases. The function must be bounded below, so that $c>0$. Furthermore, well below the nematic-isotropic transition temperature, $s=0$ should not be a local minimum of $\sigma$, hence $a<0$. A third approach came from Landau and de Gennes. Their idea was to use a matrix order parameter $\vQ:\Omega\rightarrow\mathbb{R}^{3\times 3}$ which is traceless and symmetric with the following free energy (see for example \cite{alexander2006stabilizing,kleinert1981lattice})
\begin{equation}\label{4.3}
 I_L(\vQ)=\int_\Omega |\nabla \vQ|^2+4t\vQ\cdot \nabla \times \vQ+3t^2|\vQ|^2+ \psi_B(\vQ)\,dx,
\end{equation}
where $\left( \nabla \times \vQ\right)_{ij}=\sum_{a,b} \epsilon_{jab}\vQ_{ib,a}$ and $\psi_B(\vQ):=\frac{\overline{a}}{2}\text{Tr}\left(\vQ^2\right)+\frac{\overline{b}}{3}\text{Tr}\left(\vQ^3\right)+\frac{\overline{c}}{4}\left(\text{Tr}\left(\vQ^2\right)\right)^2$. Physical constraints ensure that $\overline{c}>0$ and $\overline{a}<0$ well below the transition temperature for the same reasons as in Ericksen's theory. If in addition we impose the condition that two of the three eigenvalues of $\vQ$ must be equal, then we reduce \eqref{4.3} to the case of uniaxial states. With this restriction, $\vQ$ can be written as
\begin{equation}\label{4.4}
 \vQ:=s\left( \vn\otimes\vn-\frac{1}{3}I\right),
\end{equation}
for some $s\in \mathbb{R}$, and $\vn\in \mathbb{S}^2$. Through a simple calculation, if we substitute \eqref{4.4} into the bulk energy $\psi_B$ we get exactly the Ericksen bulk energy provided that 
\begin{equation}\label{4.5}
 \overline{a}=\frac{3a}{2},\,\,\overline{b}=\frac{9b}{2},\,\,\overline{c}=\frac{9c}{4}.
\end{equation}
Throughout this paper we will be assuming that \eqref{4.5} holds. Finally, it is clear that the general biaxial Landau-de Gennes theory cannot be equivalent to the traditional Ericksen theory \eqref{4.2}. Thus we need to define a new free energy which we will term the biaxial Ericksen free energy
\begin{equation}\label{4.6}
\begin{split}
 I_{BE}(\vn,\vm,s_1,s_2)= &
 \int_\Omega 2s_1^2\left( |\nabla\vn|^2+2t\vn\cdot\nabla \times \vn+t^2\right)+2s_2^2\left( |\nabla\vm|^2+2t\vm\cdot\nabla \times \vm+t^2\right)+\tilde{\sigma}(s_1,s_2) \\
 &-2s_1s_2\left( |\vm^T\nabla \vn-t\vn\times\vm|^2+|\vn^T\nabla\vm-t\vm\times\vn|^2-t^2\right) \\
 &+\frac{|\nabla(s_1-s_2)|^2}{3}+\frac{1}{3}\left( |\nabla s_1|^2+|\nabla s_2|^2\right)\,dx,
 \end{split}
\end{equation}
where $s_1\geqslant 0$, $s_2\leqslant 0$, and $\vn,\vm \in \mathbb{S}^2$. We have termed it the biaxial Ericksen free energy because it uses two scalar order parameters and two directors and reduces to \eqref{4.2} if $s_1=0$ or $s_2=0$. The second half of the paper focuses on variational problems for functions of bounded variation. In order to study such problems it is important to introduce some results in SBV which we will use in the direct method to prove existence of a minimiser. Suppose that $\Omega \subset \mathbb{R}^d$ is an open and bounded set.
\begin{theorem}[Closure of SBV \cite{ambrosio2000functions}]\label{4: SBV closure}
 Let $\phi:[0,\infty) \rightarrow [0,\infty]$, $\theta : (0,\infty) \rightarrow (0,\infty]$ be two lower-semicontinuous, increasing functions which satisfy 
 \begin{equation}\label{5.1}
  \lim_{t\rightarrow \infty} \frac{\phi(t)}{t} = \infty,\quad \lim_{t\rightarrow 0} \frac{\theta(t)}{t} = \infty.
 \end{equation}
 If $(u_j) \subset SBV\left(\Omega ,\mathbb{R} \right)$ is a sequence such that 
 \begin{equation}\label{5.2}
  \sup_j \left\{ \, \int_\Omega \phi\left( |\nabla u_j|\right)\,dx+ \int_{S_{u_j}} \theta\left( |u_{j+}-u_{j-}|\right) \, d\mathcal{H}^{d-1}\, \right\} <\infty.
 \end{equation}
 Then if $(u_j)$ is weak$^*$ convergent in $BV\left(\Omega,\mathbb{R}\right)$ then its limit $u \in SBV\left( \Omega ,\mathbb{R} \right)$. Additionally 
 \begin{equation}\label{5.3}
 \begin{split}
  \liminf_{j\rightarrow \infty} \int_\Omega \phi\left(|\nabla u_j|\right)\,dx \geqslant \int_\Omega \phi\left(|\nabla u|\right)\,dx \\
  \liminf_{j\rightarrow \infty} \int_{S_{u_j}} \theta\left(|u_{j+}-u_{j-}|\right)\,dx \geqslant \int_{S_u} \theta\left(|u_{+}-u_{-}|\right)\,dx
 \end{split}
 \end{equation}
 if $\phi$ is convex and $\theta$ is concave.
\end{theorem}

\begin{theorem}[Compactness of SBV \cite{ambrosio2000functions}]\label{4: SBV compactness}
 
 If $(u_j) \subset SBV\left( \Omega ,\mathbb{R} \right)$ is a sequence of functions satisfying \eqref{5.2} and $\sup_j ||u_j||_{\infty} <\infty$ then there exists some $u \in SBV\left( \Omega ,\mathbb{R}\right)$ such that for some subsequence $u_{j_k} \tostar u$ in BV.

\end{theorem}

We note that these closure and compactness notions are also true in the vectorial case. See for example \cite[Thm 2.1]{focardi2001variational}. Therefore from a mathematical point of view, functions of bounded variation form a sensible framework within which to study the calculus of variations. One final important result to introduce is that of one-dimensional sections of SBV functions  (See \cite[Remark 3.104]{ambrosio2000functions} or \cite[Section 1]{celada1997minimum}).

\begin{theorem}\cite[Section 1]{celada1997minimum}\label{4: 1D SBV sections}
 Let $U \subset \mathbb{R}^d$ be an open set, $\zeta \in \mathbb{S}^{d-1}$. Define $A:= \left\{ x \in \mathbb{R}^d \,|\, x\cdot \zeta =0\, \right\}$. If $\vn \in SBV\left( U ,\mathbb{R}^k \right)$ and $x \in A$ then 
 \begin{equation}\label{5.52}
  \vm(t) := \vn(x + t\zeta) \in SBV\left( U_{x,\zeta},\mathbb{R}^k \right),
 \end{equation}
 where $U_{x\zeta} := \left\{ \, t \in \mathbb{R} \, \left| \, x + t\zeta \in U \, \right. \right\}$. Furthermore $S_{\vm} = \left( S_\vn \right)_{x\zeta}$.

\end{theorem}

This result will be of importance when explicit minimisers are sought for specific nematic and cholesteric problems in Sections \ref{4: sec lavrentiev} and \ref{section: cholesteric problem}.

%
%
%
%
%
%

\section{Preceding Lemmas}\label{sec:lemmas}

In order to prove the equivalence results of this paper we require a number of relatively simple preliminary lemmas. Therefore we prove them in this section in isolation, then we will apply them when appropriate in the subsequent theorems. Lemmas \ref{3:lemma2}, \ref{3:lemma6}, and \ref{3:lemma3} concern specific facts about Sobolev spaces and Lemmas \ref{3:lemma4} and \ref{3:lemma5} are central to deriving the appropriate function space for the scalar order parameters. It should be noted that when we are dealing with Sobolev functions we always assume that we have chosen the precise representative amongst the equivalence class of functions equal almost everywhere. Using this precise representative \cite[Section 4.8]{evans1991measure} it is important to define a little notation that we will be using. If $f \in W^{1,p}\left( \Omega \right)$ then we define
\begin{equation}\label{4.6.1}
 C_f:= \bigcap_{\epsilon>0} \overline{\left\{  f \leqslant \epsilon \right\}}.
\end{equation}
This might just seem like an unnecessarily complicated method of defining the set where $f$ vanishes. However this is not true in general. If $f$ is continuous then $C_f = \left\{ f = 0\right\}$, but the converse is false. Clearly $C_f$ is closed and contains $\left\{ f=0 \right\}$, therefore
\begin{equation}\label{4.6d}
 \overline{ \left\{ f=0 \right\} } \subset C_f.
\end{equation}
In this paper we will need to use two different characterisations of Sobolev spaces. The first is the absolute continuity on lines (see for example \cite[Section 4.9.2]{evans1991measure}) and the second is the difference quotient characterisation \cite[Thm 10.55]{leoni2009first}.

\begin{theorem}[Absolute continuity]\label{absolutely continuous}
 Let $p\in  [1,\infty)$ and $\Omega \subset \mathbb{R}^d$ be an open set. If $f\in W^{1,p} \left( \Omega \right)$ (and is the precise representative) then for each $k=1,\dots,d$ 
 \begin{equation}\label{4.6e}
  f^*_k(x_0,t) = f(\dots,x_{k-1},t,x_{k+1},\dots)
 \end{equation}
 is absolutely continuous in $t$ for $L^{d-1}$ almost every $x_0 = (x_1,\dots,x_{k-1},x_{k+1},\dots,x_d) \in \mathbb{R}^{d-1}$. Additionally $\left(f^*_k \right)' \in L^p\left( \Omega \right)$.
 
 \vspace{3 mm}
 
 Conversely, if $f \in L^p\left( \Omega \right)$ and $f=g$ almost everywhere, where for each $k=1,\dots,d$, the functions 
 \begin{equation}\label{4.6f}
  g_k(x_0,t) = g(\dots,x_{k-1},t,x_{k+1},\dots)
 \end{equation}
 are absolutely continuous in $t$ for $L^{d-1}$ almost every $x_0 \in \mathbb{R}^{d-1}$ and $g'_k \in L^p\left( \Omega \right)$. Then $f \in W^{1,p}\left( \Omega \right)$.

\end{theorem}

\begin{proposition}[Difference quotient]\label{difference quotient}
Let $p\in [1,\infty)$ and $\Omega \subset \mathbb{R}^d$ be an open set and $u\in L^p(\Omega)$. Then 
\begin{equation}\label{4.9b}
\begin{split}
u \in W^{1,p}\left(\Omega\right) \, \Longleftrightarrow &\, \forall \, i=1,\dots,d \,\,\exists K>0 \,\, \text{such that} \,\,\forall \,\Omega'\subset\subset \Omega\,\,\text{and} \\ 
&\, |h|\leqslant dist(\Omega',\partial\Omega)\,\text{we have}\,\, ||u(\cdot+he_i)-u(\cdot)||_{p,\Omega'} \leqslant K 
\end{split}
\end{equation}

\end{proposition}

%
%
%
%
%
%

\begin{lemma}\label{3:lemma2}

Let $p\in [1,\infty)$ and $\Omega\subset \mathbb{R}^d$ be an open and bounded set. Suppose $f\in W^{1,p} \left(\Omega\right)$, and $g:\Omega \rightarrow \mathbb{R}$ are maps such that 
\begin{equation}\label{4.7}
|f(x)|\geqslant \delta >0\,\, a.e.\, x\in\Omega, \,\,\,\, fg\in W^{1,p}\left( \Omega \right)\,\,\,\, \text{and} \,\,\,\, |g(x)|\leqslant C\,\, a.e.\, x \in\Omega.
\end{equation}
Then $g\in W^{1,p} \left( \Omega\right)$.

\end{lemma}

\begin{proof}

In the spirit of using Theorem \ref{absolutely continuous}, we take some $x_0 \in \mathbb{R}^{d-1}$. Then without loss of generality we can assume that $f$ and $fg$ are absolutely continuous on the set 
\begin{equation}\label{4.8}
 U_{x_0} := \left\{ \, (t,x_0) \in \Omega \, | \, t \in \mathbb{R}\, \right\}.
\end{equation}
Hence we can use the standard facts that the reciprocal of a non-zero absolutely continuous function and the product of two absolutely continuous functions are still absolutely continuous to deduce that 
\begin{equation}\label{4.9}
 g_1(t,x_0) := \frac{(fg)(t,x_0)}{f(t,x_0)} 
\end{equation}
is absolutely continuous on $U_{x_0}$. Furthermore 
\begin{equation}\label{4.10}
 \frac{d}{dt}g_1(t,x_0) = \frac{ \frac{\partial}{\partial x_1} (fg)(x)}{f(x)} - \frac{g(x)\frac{\partial }{\partial x_1} f(x)}{f(x)}.
\end{equation}
Our assumptions imply that $\frac{1}{f},g \in L^\infty(\Omega)$ and $\nabla f, \nabla (fg) \in L^p (\Omega)$, hence \eqref{4.10} implies that
\begin{equation}\label{4.11}
 g'_1 \in L^p(\Omega).
\end{equation}
This logic applies for all line segments parallel to the coordinate axes, hence by Theorem \ref{absolutely continuous} $g\in W^{1,p} \left( \Omega \right)$.

\end{proof}

%
%
%
%
%
%

\begin{lemma}\label{3:lemma6}
 Let $p \in [1,\infty)$ and $\Omega \subset \mathbb{R}^d$ be an open set. Suppose that $f \in W^{1,p}\left( \Omega \right)$, and $g:\Omega \setminus C_f \rightarrow \mathbb{R}$ are maps such that $fg \in W^{1,p}\left( \Omega \right)$ and $|g(x)|\leqslant D$ for almost every $x \in \Omega \setminus C_f$. Then
 \begin{equation}\label{4.19.1}
  g \in W^{1,p}_{loc}\left( \Omega \setminus C_f \right).
 \end{equation}

\end{lemma}

\begin{proof}
 
 We take some $U \subset \subset \Omega \setminus C_f$ and we just need to show that $|f(x)|$ is bounded below away from zero on $U$ in order to apply Lemma \ref{3:lemma2}. This is where the importance of the definition of $C_f$ becomes clear. We can use it to deduce that
 \begin{equation}\label{4.19.3}
 {\rm ess inf}_{x \in\overline{U}}\, |f(x)| > 0.
 \end{equation}
 If this were not the case then $\overline{U} \cap \left\{ f \leqslant \epsilon \right\}$ has positive measure for all $\epsilon>0$. Then $V_\epsilon := \overline{U} \cap \overline{ \left\{ f \leqslant \epsilon \right\} }$ is a nested sequence of non-empty closed sets. Thus its intersection contains a point. This contradicts $U \subset \subset \Omega \setminus C_f$. This puts us in the position where we can simply apply Lemma \ref{3:lemma2} to deduce that 
 \begin{equation}\label{4.19.4}
  g \in W^{1,p} \left( U \right) \Longrightarrow g \in W^{1,p}_{loc} \left( \Omega\setminus C_f \right).
 \end{equation}

\end{proof}

%
%
%
%
%
%

\begin{lemma}\label{3:lemma3}
 Let $p\in [1,\infty)$ and $\Omega\subset \mathbb{R}^d$ be an open set. Suppose that $f,g:\Omega\rightarrow \mathbb{R}$ are two maps such that 
\begin{equation}\label{4.13}
f \in W^{1,p}\left(\Omega\right)\cap C(\Omega), \quad g\in W^{1,p}_{loc}\left(\Omega \setminus {\left\{f=0\right\}}\right)\,\, \text{and}\,\, fg\in W^{1,p}\left(  \Omega\setminus {\left\{f=0\right\}}\right) .
\end{equation}
Then $fg\in W^{1,p}\left(\Omega\right)$.

\end{lemma}

\begin{proof}

For ease of notation we let $\Omega':=\Omega\setminus {\left\{f=0\right\}}$. If we let $h:\Omega \rightarrow \mathbb{R}$ be defined by
\begin{equation}\label{4.14}
h(x):= \left\{ \begin{array}{lcl} f(x)g(x) && \text{if  }x\in \Omega'\\
0 && \text{otherwise} \end{array} \right. .
\end{equation}
Then clearly we have $h,\nabla h \in L^{p}\left( \Omega \right)$ where
\begin{equation}\label{4.15}
\nabla h :=  \left\{ \begin{array}{lcl} g(x)\nabla f(x) + f(x)\nabla g(x) && \text{if  }x\in \Omega'\\
0 && \text{otherwise} \end{array} \right. .
\end{equation}
Therefore we just need to show that the integration by parts formula holds in order to deduce the result. We take a $\psi \in C^\infty_0 \left( \Omega \right)$ and we need to show that 
\begin{equation}\label{4.16}
\int_{\Omega'} f(x)g(x)\frac{\partial \psi}{\partial x_j}(x)\,dx=-\int_{\Omega'} \frac{\partial}{\partial x_j}\left[f(x)g(x)\right]\psi(x)\, dx.
\end{equation}
We know that Sobolev functions are absolutely continuous on almost every line segment parallel to the coordinate axes (Theorem \ref{absolutely continuous}). Therefore our assumptions imply that $fg$ is absolutely continuous on $\Omega'_{y_0}:=\Omega'\cap \left\{\, (x,y_0)\,|\,x\in \mathbb{R}\, \right\}$ for almost every $y_0 \in \mathbb{R}^{d-1}$. Then by splitting up $\Omega'_{y_0}$ into its connected components $U_i$, which will be open 1-dimensional intervals of the form $(a_i,b_i)$, we see that 
\begin{equation}\label{4.17}
\begin{split}
\int_{\Omega'_{y_0}} \left( fg\right) \psi_{,j}\,dx & = \sum_i \int_{U_i} \left( fg \right) \psi_{,j}\,dx \\
&= -\int_{\Omega'_{y=0}}\left(fg \right)_{,j} \psi\,dx+\sum_i \left[fg\psi(b_i,y_0)-fg\psi(a_i,y_0) \right].
\end{split} 
\end{equation}
Here we have two possibilities. If $(a_i,y_0)\in \partial\Omega$ then $\psi(a_i,y_0)=0$ since $\psi \in C^\infty_0\left( \Omega \right)$. Alternatively if $(a_i,y_0)\in \partial {\left\{ f=0 \right\}}$ then $f(a_i,y_0)=0$. Therefore in either case we see that each element of the sum in \eqref{4.17} must be zero. Hence
\begin{equation}\label{4.18}
\int_{\Omega'_{y_0}} \left( fg\right) \psi_{,j}\,dx = -\int_{\Omega'_{y=0}}\left(fg \right)_{,j} \psi\,dx,
\end{equation}
for almost every $y_0 \in \mathbb{R}^{d-1}$, so by performing the remaining $d-1$ integrations we immediately find \eqref{4.16}. Therefore we have shown
\begin{equation}\label{4.19}
 \int_\Omega h\psi_{,j} \,dx=-\int_\Omega h_{,j}\psi\,dx \quad \forall \psi \in C^\infty_0(\Omega), \,\, j=1,\dots,n.
\end{equation}
Hence $h\in W^{1,p}\left( \Omega \right)$ and $h=fg$ $\mathcal{L}^n$ almost everywhere which gives us our conclusion.

\end{proof}

%
%
%
%
%
%

\begin{lemma}\label{3:lemma4}
Using the Frobenius matrix norm $|A|^2 = \sum_{i,j} a_{ij}^2$, we have the following inequality.
\begin{equation}\label{4.20}
\left|s_1\left( \vn\otimes\vn-\frac{1}{3}Id\right)-s_2 \left( \vm\otimes\vm-\frac{1}{3}Id \right) \right|^2\geqslant \frac{1}{6}\left| s_1-s_2 \right|^2
\end{equation}
for any $s_1,s_2 \in \mathbb{R}$ and $\vn,\vm \in \mathbb{S}^2$.

\end{lemma}

\begin{proof}

The proof is simply by direct calculation. When we multiply out the left hand side of \eqref{4.20} we obtain
\begin{equation}\label{4.21}
\begin{split}
LHS
&=\sum_{i,j} s_1^2\left( \vn_i\vn_j-\frac{1}{3}\delta_{i,j}\right)\left( \vn_i\vn_j-\frac{1}{3}\delta_{i,j}\right)+s_2^2\left( \vm_i\vm_j-\frac{1}{3}\delta_{i,j}\right)\left( \vm_i\vm_j-\frac{1}{3}\delta_{i,j}\right)\\
&-\sum_{i,j} 2s_1s_2\left( \vn_i\vn_j-\frac{1}{3}\delta_{i,j}\right)\left( \vm_i\vm_j-\frac{1}{3}\delta_{i,j}\right) \\
& = \frac{2}{3}\left(s_1^2+s_2^2\right) -2s_1s_2\left( (\vn\cdot\vm)^2-\frac{1}{3}\right).
\end{split}
\end{equation}
Now we split the proof into the two cases of $s_1s_2\geqslant 0$ and $s_1s_2<0$ and show that we have the result in either situation. If we are in the first case then
\begin{equation}\label{4.22}
LHS\geqslant \frac{2}{3}\left( s_1^2+s_2^2 \right) -\frac{4s_1s_2}{3}=\frac{2}{3}\left( s_1-s_2\right)^2\geqslant \frac{1}{6} \left( s_1-s_2\right)^2.
\end{equation}
In the second case we have 
\begin{equation}\label{4.23}
\begin{split}
LHS&
\geqslant \frac{2}{3}\left(s_1^2+s_2^2\right) +\frac{2s_1s_2}{3} \\
&=\frac{2}{3}\left(s_1^2+s_2^2\right) +\frac{2s_1s_2}{3}-\frac{1}{6}\left(s_1-s_2\right)^2+\frac{1}{6}\left(s_1-s_2\right)^2\\
&=\frac{1}{2} \left(s_1+s_2\right)^2+\frac{1}{6}\left(s_1-s_2\right)^2\\
&\geqslant \frac{1}{6}\left(s_1-s_2\right)^2.
\end{split}
\end{equation}

\end{proof}

%
%
%
%
%
%

\begin{lemma}\label{3:lemma5}
 Suppose that 
 \begin{equation}\label{4.24}
  \begin{split}
   \vQ_1=s_1\left( \vn_1\otimes \vn_1-\frac{1}{3}I\right)+s_2\left( \vn_2\otimes \vn_2-\frac{1}{3}I\right) \\
   \vQ_2=t_1\left( \vm_1\otimes \vm_1-\frac{1}{3}I\right)+t_2\left( \vm_2\otimes \vm_2-\frac{1}{3}I\right)
  \end{split}
 \end{equation}
 for $s_1,t_1\geqslant 0$, $s_2,t_2\leqslant 0$ and $\vn_1,\vn_2,\vm_1,\vm_2 \in \mathbb{S}^2$. Then 
 \begin{equation}\label{4.25}
  |\vQ_1-\vQ_2|^2 \geqslant \frac{1}{3}\left[(s_1-t_1)^2(s_2-t_2)^2\right]+s_1t_1\left|\vn_1\otimes \vn_1-\vm_1\otimes \vm_1\right|^2+s_2t_2 \left|\vn_2\otimes\vn_2-\vm_2\otimes\vm_2\right|^2.
 \end{equation}

\end{lemma}

\begin{proof}
 
 As in the previous lemma, \eqref{4.25} is understood using the Frobenius matrix norm $|A|^2=\sum_{i,j}a_{i,j}^2$. We multiply out the left hand side of \eqref{4.25} to find 
 \begin{equation}\label{4.26}
  \begin{split}
   \left| \vQ_1-\vQ_2\right|^2 &
   = \sum_{i,j} \left[ s_1\left(\vn_{1 i}\vn_{1 j}-\frac{1}{3}\delta_{i,j}\right) + s_2\left(\vn_{2 i}\vn_{2 j}-\frac{1}{3}\delta_{i,j}\right) - t_1\left(\vm_{1 i}\vm_{1 j}-\frac{1}{3}\delta_{i,j}\right) - t_2\left(\vm_{2 i}\vm_{2 j}-\frac{1}{3}\delta_{i,j}\right)\right]^2 \\
   &=\frac{2}{3} \left( s_1^2+s_2^2+t_1^2+t_2^2\right)-\frac{2}{3}(s_1s_2+t_1t_2)-2\sum_{i,j=1}^2 s_it_j\left( (\vn_i\cdot\vm_j)^2-\frac{1}{3}\right).
  \end{split}
 \end{equation}
 However we know that $s_1t_2,s_2t_1\leqslant 0$ and for any two unit vectors $\vn$ and $\vm$, $|\vn\otimes\vn-\vm\otimes\vm|^2=2-2(\vn\cdot\vm)^2$. These two facts allow us to estimate the above expression and complete the assertion with the following calculation
 \begin{equation}\label{4.27}
  \begin{split}
   \left| \vQ_1-\vQ_2\right|^2 & 
   \geqslant \frac{2}{3}\left( s_1^2+s_2^2+t_1^2+t_2^2-s_1s_2-t_1t_2\right)+\frac{2}{3}\left( s_1t_2+s_2t_1\right) \\
   &- 2s_1t_1\left( (\vn_1\cdot\vm_1)^2-\frac{1}{3}\right)-2s_2t_2\left( (\vn_2\cdot\vm_2)^2-\frac{1}{3}\right) \\
   & = \frac{2}{3}\left( s_1^2+s_2^2+t_1^2+t_2^2-s_1s_2-t_1t_2+s_1t_2 + s_2t_1 -2s_1t_1 -2s_2t_2\right) \\
   &+s_1s_2\left|\vn_1\otimes\vn_1-\vm_1\otimes\vm_1\right|^2+t_1t_2|\vn_2\otimes\vn_2-\vm_2\otimes\vn_2|^2 \\
   &= \frac{1}{3}\left( (s_1-t_1)^2+(s_2-t_2)^2+(s_1-s_2-t_1+t_2)^2\right)\\
   &+s_1s_2\left|\vn_1\otimes\vn_1-\vm_1\otimes\vm_1\right|^2+t_1t_2|\vn_2\otimes\vn_2-\vm_2\otimes\vn_2|^2 \\
   &\geqslant \frac{1}{3}\left[ (s_1-t_1)^2+(s_2-t_2)^2\right]+s_1t_1\left|\vn_1\otimes \vn_1-\vm_1\otimes \vm_1\right|^2+s_2t_2 \left|\vn_2\otimes\vn_2-\vm_2\otimes\vm_2\right|^2.
  \end{split}
 \end{equation}

\end{proof}

%
%
%
%
%
%

%
%
%
%
%
%

\section{One-dimensional Results}\label{sec:1D results}

We are now in a position to begin proving our equivalence results. However we must first prove the relevant results about the lifting of line fields to vector fields in one dimension. As the following proposition makes clear, in one-dimensional Sobolev spaces we can find a vector field lifting with the same regularity as the original line field. For the lifting results in simply-connected domains in higher dimensions, we will be building on the results of Ball \& Zarnescu \cite{ball2011orientability}.

\begin{proposition}\label{3:prop1}
 Suppose $\Omega\subset \mathbb{R}$ is a bounded domain and that $p\in [1,\infty)$. If 
 \begin{equation}\label{4.28}
  \vn\otimes\vn \in W^{1,p}_{loc}\left( \Omega,\mathbb{R}^{3\times 3}\right) 
 \end{equation}
 where $|\vn|=1$ then there exists some 
 \begin{equation}\label{4.29}
  \vm \in W^{1,p}_{loc}\left( \Omega,\mathbb{S}^2\right) 
 \end{equation}
 such that $\vm\otimes \vm=\vn \otimes \vn$ on $\Omega$.

\end{proposition}

\begin{proof}
 
 Through the Sobolev embedding theorem our assumptions imply
 \begin{equation}\label{4.30}
  \vn\otimes \vn \in C\left( \overline{U},\mathbb{R}^{3\times 3}\right).
 \end{equation}
 for any open set $U\subset \subset \Omega$. This certainly implies that $\vn\otimes \vn \in C\left(\Omega,\mathbb{R}^{3\times 3}\right)$. Therefore if we take some $x\in \Omega$, it is contained in some $U\subset\subset \Omega$. Then we choose $\vm(x)$ to be one of the two possible unit vectors to ensure that $\vm(x)\otimes \vm(x)=\vn\otimes \vn(x)$. The uniform continuity of $\vn \otimes \vn$ on $\overline{U}$ then ensures that this one choice at $x\in \Omega$ defines a mapping $\vm \in C\left( \overline{U},\mathbb{S}^2\right)$. Hence
 \begin{equation}\label{4.31}
  \vm \in C\left( \Omega,\mathbb{S}^2\right).
 \end{equation}
To prove its differentiability, we note that for $h>0$ we have the relation 
\begin{equation}\label{4.32}
\begin{split}
 &\left[ \frac{\vm_i(x+h)\vm_j(x+h)-\vm_i(x)\vm_j(x)}{h}\right]\left[ \frac{\vm_j(x+h)+\vm_j(x)}{2}\right] \\
 =& \frac{\vm_i(x+h)}{2}\left[ \frac{ \vm_j(x+h)^2-\vm_j(x)^2}{h}\right] +\vm_j(x)\left[ \frac{\vm_i(x+h)-\vm_i(x)}{h}\right]\left[\frac{\vm_j(x+h)+\vm_j(x)}{2}\right].
 \end{split}
\end{equation}
If we sum this expression over $j$ and take the limit as $h\rightarrow 0$, we find
\begin{equation}\label{4.33}
 \lim_{h\rightarrow 0} \frac{\vm_i(x+h)-\vm_i(x)}{h}=\sum_j (\vm_i\vm_j)'(x)\vm_j(x) = \sum_j (\vn_i\vn_j)'(x)\vm_j(x).
\end{equation}
Therefore the derivative of $\vm_i$ exists almost everywhere and is equal to  $\sum_j (\vn_i \vn_j)'(x)\vm_j(x) \in L^p_{loc}(\Omega)$, which implies 
\begin{equation}\label{4.34}
 \vm \in W^{1,p}_{loc}\left( \Omega,\mathbb{S}^2\right).
\end{equation}
 
\end{proof}

\begin{corollary}\label{3: corollary 1}
 Suppose $\Omega \subset \mathbb{R}$ is a bounded domain and that $p\in [1,\infty)$. If $\vn \otimes \vn \in W^{1,p}\left( \Omega ,\mathbb{R}^{3\times 3} \right)$ where $|\vn|=1$. Then there exists some $\vm \in W^{1,p}\left( \Omega ,\mathbb{S}^2 \right)$ such that $\vm\otimes \vm = \vn \otimes \vn$.
\end{corollary}

This corollary immediately follows from the proof of Proposition \ref{3:prop1}. This lifting result together with the lemmas from section \ref{sec:lemmas} mean we can now move ahead and prove our two equivalence results in one dimension.

%
%
%
%
%
%

\subsection{Uniaxial Landau-de Gennes}

\begin{theorem}\label{3:1D-Uniaxial-Ericksen}
Suppose $\Omega\subset \mathbb{R}$ is a bounded domain. Then minimising $I_L(\vQ)$ over 
\begin{equation}\label{4.35}
 \mathcal{A}:=\left\{\,\left. \vQ \in W^{1,2}\left( \Omega,\mathbb{R}^{3\times 3} \right) \,\right|\, \vQ^T = \vQ,\,\, {\rm Tr}\,\vQ=0,\,\, \vQ\, \text{uniaxial}\,\right\},
\end{equation}
is equivalent to minimising $I_E(\vn,s)$ over 
\begin{equation}\label{4.36}
 \mathcal{B}:=\left\{  s\in W^{1,2}\left( \Omega\right),\,\, \vn\in W^{1,2}_{loc}\left( \Omega\setminus \left\{s=0\right\},\mathbb{S}^2\right)\,\left| \, \int_\Omega s^2|\nabla \vn|^2 \, dx < \infty \right. \right\}.
\end{equation}

\end{theorem}

\begin{proof}
 
 The main idea of the proof, as with each of the subsequent equivalence results, is to show a one-to-one correspondence between the sets of admissible functions. A quick formal calculation can then show that the two Lagrangians are in fact equal so that the minimisation problems are equivalent. We begin by taking some $\vQ \in \mathcal{A}$. Then we know that this matrix has the form 
 \begin{equation}\label{4.37}
  \vQ(x)=s(x)\left( \vn\otimes\vn(x) -\frac{1}{3}I\right) 
 \end{equation}
 for some $|\vn|=1$. We apply Lemma \ref{3:lemma4} to deduce that 
 \begin{equation}\label{4.38}
  |\vQ(x)-\vQ(y)|^2\geqslant \frac{1}{6} |s(x)-s(y)|^2,
 \end{equation}
 for any $x,y \in \Omega$. This immediately implies that $s\in L^2\left( \Omega\right)$ and to show it is in the required Sobolev space we use the difference quotient characterisation of $W^{1,2}$. Equation \eqref{4.38} implies 
 \begin{equation}\label{4.39}
  ||\vQ(\cdot+h)-\vQ(\cdot)||_{2,U}\geqslant 6^{-\frac{1}{2}}||s(\cdot+h)-s(\cdot)||_{2,U}
 \end{equation}
 for any $U\subset \subset \Omega$. Therefore by Proposition \ref{difference quotient} we infer that $s\in W^{1,2}\left( \Omega\right)$. Coupling this fact with \eqref{4.37} directly implies that $s\vn\otimes\vn\in W^{1,2}\left( \Omega,\mathbb{R}^{3\times 3}\right)$. If we take any $U\subset \subset \Omega\setminus \left\{ s=0\right\}$ then by the continuity of $s$, it is clear that
 \begin{equation}\label{4.40}
  \inf_{x\in U}|s(x)|>0.
 \end{equation}
 This means that we can apply Lemma \ref{3:lemma2} to $s\vn\otimes\vn$ and deduce 
 \begin{equation}\label{4.41}
  \vn\otimes\vn \in W^{1,2}\left( U,\mathbb{R}^{3\times 3}\right).
 \end{equation}
 Therefore $\vn \otimes\vn \in W^{1,2}_{loc}\left( \Omega \setminus \left\{ s=0\right\},\mathbb{R}^{3\times 3}\right)$. Now we can apply the previous result, Proposition \ref{3:prop1}, to find the vector field corresponding to $\vn\otimes\vn$. There exists some
 \begin{equation}\label{4.42}
  \vm \in W^{1,2}_{loc}\left( \Omega\setminus \left\{ s=0\right\},\mathbb{S}^2\right) 
 \end{equation}
 such that $\vm\otimes \vm=\vn\otimes\vn$ almost everywhere. To show the joint regularity we note that $s(\vn \otimes \vn) = s (\vm \otimes \vm) \in W^{1,2}(\Omega,\mathbb{R}^{3\times 3} )$, hence
 \begin{equation}\label{4.42.1}
  \int_\Omega \left|\nabla \left[ s(\vm \otimes \vm) \right] \right|^2\, dx = \int_\Omega 2s^2 |\nabla \vm |^2 + |\nabla s|^2 \, dx < \infty
 \end{equation}
 This concludes the first inclusion; we took some $\vQ \in \mathcal{A}$ and found a pair $(s,\vm)\in \mathcal{B}$ such that $\vQ=s\left( \vm\otimes \vm -\frac{1}{3} I\right)$ almost everywhere. The reverse inclusion is a little more straightforward, we take some $(s,\vn)\in \mathcal{B}$ and just need to show that $s\vn\otimes\vn \in W^{1,2}\left( \Omega,\mathbb{R}^{3\times 3}\right)$. Since $W^{1,2}_{loc}\cap L^\infty$ is an algebra of functions, it is closed under products, so that 
 \begin{equation}\label{4.43}
  \vn\otimes\vn \in W^{1,2}_{loc}\left(\Omega\setminus \left\{ s=0 \right\},\mathbb{R}^{3\times 3}\right).
 \end{equation}
 Using the joint regularity assumption we know that for any open set $U\subset\subset \Omega\setminus \left\{ s=0\right\}$ 
 \begin{equation}\label{4.45}
  \int_{U} |\nabla (s\vn\otimes\vn)|^2\,dx =\int_{U} 2s^2|\nabla \vn|^2+|\nabla s|^2\, dx\leqslant \int_\Omega 2s^2|\nabla \vn|^2+|\nabla s|^2\,dx<\infty.
 \end{equation}
As the upper bound in \eqref{4.45} is independent of $U$, we have shown that $s\vn\otimes\vn \in W^{1,2}\left( \Omega \setminus \left\{ s=0\right\},\mathbb{R}^{3\times 3}\right)$. Finally we apply Lemma \ref{3:lemma3} to deduce $s\vn\otimes\vn \in W^{1,2}\left(\Omega,\mathbb{R}^{3\times 3}\right)$. This establishes the one-to-one correspondence between $\mathcal{A}$ and $\mathcal{B}$. For the formal equivalence of the Lagrangians we just need to take some $\vQ \in \mathcal{A}$ and compute its Lagrangian in terms of $s$ and $\vn$.
\begin{equation}\label{4.46}
 \begin{split}
  =& |\nabla \vQ|^2+4t\vQ \cdot \nabla \times \vQ+3t^2|\vQ|^2 \\
  =& \sum_{i,j,k} \left(s_{k}\left(\vn_i\vn_j- \frac{1}{3}\delta_{ij}\right)+s\left( \vn_{i,k}\vn_j+\vn_i\vn_{j,k}\right) \right)^2 \\
  +& 4t\sum_{i,j,a,b} \epsilon_{jab} s\left(\vn_i\vn_j -\frac{1}{3}\delta_{ij}\right)\left[ s_a\left(\vn_i\vn_b-\delta_{ib}\right)+s\left( \vn_{i,b}\vn_{j}+\vn_{i}\vn_{j,b}\right)\right] \\
  +&2t^2\sum_{i,j} s^2\left(\vn_i\vn_j-\frac{1}{3}\delta_{ij}\right)^2 \\
  =&2s^2\left( |\nabla\vn|^2+2t\vn\cdot\nabla\times\vn+t^2\right) +\frac{2}{3}|\nabla s|^2
 \end{split}
\end{equation}
Therefore if we have some $\vQ\in \mathcal{A}$ we have the relation that $I_L(\vQ)=I_E(\vn,s)$ if $(\vn,s)$ is the pair of functions in $\mathcal{B}$ associated with $\vQ$. Therefore any $\vQ\in \mathcal{A}$ which minimises $I_L$ will yield a pair $(\vn,s)\in \mathcal{B}$ which minimises $I_E$ and vice versa. In this sense the two minimisation problems are equivalent.

\end{proof}

%
%
%
%
%
%

\subsection{Biaxial Landau-de Gennes}

\begin{theorem}\label{3:1D-Biaxial-Ericksen}
 Suppose $\Omega\subset \mathbb{R}$ is a bounded domain. Then minimising $I_L(\vQ)$ over 
 \begin{equation}\label{4.47}
  \mathcal{A}:=\left\{ \left.\vQ \in W^{1,2}\left( \Omega,\mathbb{R}^{3\times 3}\right)\,\right|\, \vQ^T=\vQ,\,\, \text{Tr}\,\vQ=0\,\right\}
 \end{equation}
 is equivalent to minimising $I_{BE}(\vn,\vm,s_1,s_2)$ over
 \begin{equation}\label{4.48}
  \mathcal{B}:=\left\{
  \left.
  \begin{array}{c}s_1 \in W^{1,2}\left( \Omega,[0,\infty)\right),\,\, \vn\in W^{1,2}_{loc}\left( \Omega\setminus \left\{ s_1=0\right\},\mathbb{S}^2\right),\, \\
  s_2 \in W^{1,2}\left( \Omega,(-\infty,0]\right), \,\,\vm \in W^{1,2}_{loc}\left( \Omega\setminus \left\{ s_2=0\right\},\mathbb{S}^2\right)\,\end{array}
  \right|\, \int_\Omega s_1|\nabla \vn|^2 + s_2 |\nabla \vm|^2 \, dx < \infty \, 
  \right\}.
 \end{equation}

\end{theorem}

\begin{proof}
 
 The basic form of this proof follows that of the previous theorem but using more general biaxial forms. We take some $\vQ \in \mathcal{A}$ then we know that we can write it as 
 \begin{equation}\label{4.49}
  \vQ=s_1\left( \vn\otimes \vn -\frac{1}{3}I\right)+s_2\left( \vm\otimes\vm-\frac{1}{3}I\right)
 \end{equation}
 for $\vn,\vm \in \mathbb{S}^2$, $s_1\geqslant 0$ and $s_2\leqslant 0$. Applying Lemma \ref{3:lemma5} to this function $\vQ(x)$, we deduce that 
 \begin{equation}\label{4.50}
 \begin{split}
 |\vQ(x)-\vQ(y)|^2 &\geqslant \frac{1}{3}\left[ (s_1(x)-s_1(y))^2+(s_2(x)-s_2(y))^2\right] \\
 &+s_1(x)s_1(y)\left| \vn\otimes\vn(x)-\vn\otimes\vn(y)\right|^2 +s_2(x)s_2(y)\left| \vm\otimes\vm(x)-\vm\otimes\vm(y)\right|^2.
 \end{split}
 \end{equation}
 As in the previous proof, using the difference quotient characterisation of Sobolev spaces with \eqref{4.50} we find 
 \begin{equation}\label{4.51}
  s_1,s_2\in W^{1,2}\left( \Omega\right).
 \end{equation}
 In order to find the correct function space for $\vn$ we take some open set $U\subset \subset \Omega\setminus\left\{ s_1=0\right\}$. There exists some $\epsilon>0$ such that 
 \begin{equation}\label{4.52}
  \inf_{x\in U}|s_1(x)|\geqslant \epsilon.
 \end{equation}
 This means that $\inf_{x,y\in U} s_1(x)s_1(y)\geqslant \epsilon^2$. In other words, we deduce from \eqref{4.50} that for $x,y\in U$
 \begin{equation}\label{4.53}
  |\vQ(x)-\vQ(y)|\geqslant \epsilon^2 |\vn\otimes\vn(x)-\vn\otimes\vn(y)|^2.
 \end{equation}
Again, using the difference quotient characterisation with this equation implies $\vn \otimes\vn \in W^{1,2}\left( U,\mathbb{R}^{3\times3}\right)$. As $U$ was arbitrary this tells us that $\vn\otimes\vn\in W^{1,2}_{loc}\left( \Omega\setminus \left\{ s_1=0\right\},\mathbb{R}^{3\times3}\right)$. We apply Proposition \ref{3:prop1} to the line field to deduce the existence of some 
\begin{equation}\label{4.54}
 \vn_1 \in W^{1,2}_{loc}\left( \Omega\setminus \left\{ s_1=0\right\},\mathbb{S}^2\right) 
\end{equation}
such that $\vn_1\otimes\vn_1=\vn\otimes\vn$ whenever $s_1\neq0$. By identical reasoning we can also find a mapping $\vm_1 \in W^{1,2}\left( \Omega\setminus \left\{ s_2=0 \right\},\mathbb{S}^2\right) $ such that $\vm_1\otimes\vm_1=\vm\otimes\vm$ whenever $s_2\neq0$. The joint regularity condition for the two pairs of functions is clear by using the same logic as from \eqref{4.42.1}. This concludes the first half of the correspondence between the admissible sets of functions. For the reverse inclusion we take a group of mappings $(\vn,\vm,s_1,s_2)\in \mathcal{B}$ and we just need to show that $s_1\vn\otimes\vn,s_2\vm\otimes\vm\in W^{1,2}\left( \Omega,\mathbb{R}^{3\times 3}\right)$ to ensure that 
\begin{equation}\label{4.55}
 \vQ:=s_1\left( \vn\otimes\vn -\frac{1}{3}I\right) +s_2\left( \vm\otimes\vm-\frac{1}{3}I\right)\in W^{1,2}\left( \Omega,\mathbb{R}^{3\times3}\right).
\end{equation}
However we can use the same logic from the previous proof (see \eqref{4.43} and \eqref{4.45}) to each of the pairs $(s_1,\vn)$ and $(s_2,\vm)$ separately to arrive at \eqref{4.55}. This proves the equivalence between the sets $\mathcal{A}$ and $\mathcal{B}$. It is a simple but long calculation to show that if we have some $\vQ$ in the form of \eqref{4.55}, then substituting this into the Landau-de Gennes energy \eqref{4.3} will precisely produce the biaxial Ericksen energy \eqref{4.6}. For the sake of brevity and elegance we do not include the calculation here. Therefore the proof is complete.
 
\end{proof}

%
%
%
%
%
%

\section{Two and Three-dimensional Results}\label{sec:2,3D results}

 In higher dimensions there are a greater number of intricacies, primarily because of the fact that not all connected domains are simply connected, unless we are in dimension one. The importance of simply-connected domains when looking to turn a line field into a vector field is crucial. On the simplest level, given a line field $\vn\otimes\vn$, at every point of the domain there are two possible choices for the vector field: $\vn$ and $-\vn$. A choice is made at some given point $x\in \Omega$ and we want to use this decision to choose what the vector field should be at every other point. This is more easily done is simply-connected domains because any two continuous curves in $\Omega$ connecting two points $x,y\in \Omega$ are homotopic. As an example to illustrate this, in Figure \ref{fig:athletics_track_example} the smooth line field as denoted by the dashed lines cannot be converted into a smooth vector field because of the holes in the domain.

 \vspace{3 mm}
 
We will prove three lifting results analogous to Proposition \ref{3:prop1} about the regularity of line fields and vectors fields in two and three dimensions. For the first we extend a result of Ball and Zarnescu.

\begin{theorem}\cite[Theorem 2]{ball2011orientability}
 Suppose that $d=2\,\,\text{or}\,\,3$ and $\Omega\subset \mathbb{R}^d$ is an open, bounded and simply connected set with continuous boundary. Suppose there is a mapping $\vn:\Omega\rightarrow \mathbb{S}^2$ such that $\vn\otimes \vn\in W^{1,2}\left(\Omega,\mathbb{R}^{3\times3}\right)$. Then
\begin{equation}\label{4.55.0}
  \exists\, \vm\in W^{1,2}\left(\Omega,\mathbb{S}^{2}\right)\,\,\text{s.t.}\,\,\vm\otimes\vm=\vn\otimes\vn\,\,\text{a.e.}
\end{equation}
\end{theorem}

This result assumes that the simply connected domain has a continuous boundary. Our extension will remove this assumption. Therefore we will show that it is possible to convert line fields to vector fields in the space $W^{1,2}$ when the domain is simply connected.

\begin{figure}[ht]
\centering
\includegraphics[scale=0.3]{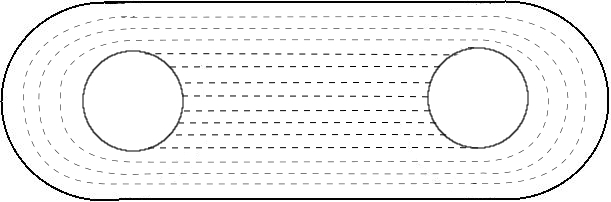}
\caption{Non-simply connected domain example}
\label{fig:athletics_track_example}
\end{figure}

\begin{proposition}\label{3:prop2 Ball-Zarnescu Extension}
 Suppose that $d=2\,\,\text{or}\,\,3$ and $\Omega\subset \mathbb{R}^d$ is an open, bounded and simply connected set. Suppose there is a mapping $\vn:\Omega\rightarrow \mathbb{S}^2$ such that $\vn\otimes \vn\in W^{1,2}\left(\Omega,\mathbb{R}^{3\times3}\right)$. Then
\begin{equation}\label{4.56}
  \exists\, \vm\in W^{1,2}\left(\Omega,\mathbb{S}^{2}\right)\,\,\text{s.t.}\,\,\vm\otimes\vm=\vn\otimes\vn\,\,\text{a.e.}
\end{equation}

\end{proposition}

\begin{proof}

Since $\Omega$ is an open set we know that for all $x\in\Omega$ there exists a maximal radius $\epsilon_x>0$, such that 
\begin{equation}\label{4.57}
 U_x :=B(x,\epsilon_x)\subseteq \Omega.
\end{equation}
The result of Ball-Zarnescu \cite[Theorem 2]{ball2011orientability} applies in any simply connected domain which has a continuous boundary, so certainly we can apply their result to each of these balls $U_x$ to deduce that 
\begin{equation}\label{4.58}
 \exists \,\vm\in W^{1,2}(U_x,\mathbb{S}^{2})\,\,\text{s.t.}\,\,\vm\otimes\vm=\vn\otimes\vn\,\,\text{a.e. on}\,\,U_x.
\end{equation}
We also know that whenever we apply this result there are two possible choices for the vector field: $\vm$ and $-\vm$. The main issue in this proof is how to choose the orientations in each ball $U_x$ so that at the finish we have a function with $W^{1,2}$ regularity on the whole of $\Omega$.

\vspace{3 mm}

We fix an arbitrary point $x\in\Omega$. Applying \cite[Theorem 2]{ball2011orientability} precisely gives us the statement \eqref{4.58}. Now we proceed to show that this one choice uniquely defines the choice that we must make in every other ball $U_y$ in $\Omega$. We take $y\in\Omega$, $y\neq x$, and we suppose we have a continuous injective curve $f$ such that
\begin{equation}\label{4.59}
 f:[0,1] \rightarrow \Omega\,\,\text{and} \,\,f(0)=x,\,\,f(1)=y.
\end{equation}
We use this curve to motivate the following definition. For $t\in [0,1]$ let
\begin{equation}\label{4.60}
 t_+:=\inf\left\{\,\mu\geq t \,\left|\,\mu\in[0,1],\,\,f(\mu)\notin U_{f(t)}\,\right\}\right.\wedge\left\{1\right\}.
\end{equation}
This simply defines the first point on the curve to leave the set $U_{f(t)}$ (or $1$ if there is no such point). Using this notion we define a sequence $(t_j)$ as follows:
\begin{equation}\label{4.61}
 t_1:=0_+\quad t_2:=\left(t_1\right)_+ \quad t_3:=\left(t_2\right)_+ \,\, \dots
\end{equation}

\begin{lemma}\label{3:prooflemma1}
$\exists \, N\in\mathbb{N}$ such that $t_N =1$.
\end{lemma}

\begin{proof}

It is clear from \eqref{4.60} that if $t \in (0,1)$ then $t<t_+\leq 1$ so that this sequence is a monotonically increasing sequence which is bounded above by $1$. Such a sequence must converge. For a contradiction we suppose
\begin{equation}\label{4.62}
 t_j \rightarrow t <1.
\end{equation}
Consider the set $U_{f(t)}$; we know that $\epsilon_{f(t)}>0$. We also know $f$ is a continuous function so $f(t_j)\rightarrow f(t)$. Hence 
\begin{equation}\label{4.63}
\exists \, J\,\,\text{s.t.}\,\,\forall j\geq J\quad f(t_j)\subset B\left( f(t),\frac{\epsilon_{f(t)}}{2}\right).
\end{equation}
As we defined the radii of our balls to be maximal it must be that 
\begin{equation}\label{4.64}
 \epsilon_{f(t_J)}>\frac{\epsilon_{f(t)}}{2},
\end{equation}
thus 
\begin{equation}\label{4.65}
 t_{J+1}=\left(t_J\right)_+>t,
\end{equation}
which is a contradiction. So we must have that $t_j\rightarrow 1$ as $j\rightarrow \infty$. Using this argument again we deduce that there can only be finitely many different terms. We know
\begin{equation}\label{4.66}
 \exists \, M\,\,\text{s.t.}\,\,\forall j\geq M\quad f(t_j)\subset B\left( f(1),\frac{\epsilon_{f(1)}}{2}\right)=B\left( y,\frac{\epsilon_{y}}{2}\right).
\end{equation}
This time maximality allows us to deduce that it must be the case that $\left(t_M\right)_+=t_{M+1}=1$.

\end{proof}

\vspace{3 mm}

It is also necessary to note a second property before we continue onto the main portion of the proof.

\begin{lemma}\label{3:prooflemma2}
 \begin{equation}\label{4.67}
  \inf_{t \in [0,1]} \epsilon_{f(t)} >0
 \end{equation}
\end{lemma}

\begin{proof}

This statement is simple because the explicit form for $\epsilon_x$ is $d(x,\partial\Omega)$. The continuity of the distance function ensures that $\epsilon_x$ is a continuous function of $x$. Combining this with the fact that it is positive on the compact set $f([0,1])$, gives \eqref{4.67}. 

\end{proof}
 
\vspace{3 mm}

Now that we have established these lemmas we can move on to the choice construction. For all $t\in[0,t_1]$ we know
\begin{equation}\label{4.68}
 \mathcal{L}^d\left( U_x \cap U_{f(t)}\right)>0,
\end{equation}
so that we can uniquely choose the orientation in $U_{f(t)}$ by ensuring that on the set $U_x \cap U_{f(t)}$, it agrees with the orientation already chosen for $U_x$. We continue this process for the finite number of intervals $[t_j,t_{j+1}]$. This construction ensures that we have uniquely chosen an orientation at each point along the curve using the orientation at its starting point. Now we need to show that the orientation chosen for $U_y$ is independent of the path chosen. To do this we take a second continuous injective curve $g:[0,1]\rightarrow \Omega$ such that $g(0)=x$ and $g(1)=y$. At this juncture we use the crucial assumption that $\Omega$ is simply connected which has the following formulation.

\vspace{3 mm}

\textbf{\emph{A set is simply connected if and only if it is path connected and any two continuous curves with the same endpoints are homotopic.}}

\vspace{3 mm}

This means that our two curves $f$ and $g$ must be homotopic in the sense that there exists a continuous map $H$, with the following properties
\begin{equation}\label{4.69}
 H:[0,1]\times[0,1]\rightarrow \Omega \quad\text{s.t.}\quad H(t,0)=f(t),\,\,H(t,1)=g(t),\,\,H(0,\lambda)=x\,\,\text{and}\,\,H(1,\lambda)=y.
\end{equation}
For a contradiction we assume that using this second path we arrive at the other orientation on the set $U_y$. Then we can define
\begin{equation}\label{4.70}
 \lambda^*:= \inf\left\{\, \mu\,|\,H(1,\mu)\,\,\text{gives the opposite orientation to }H(1,0)=f(1)\,\right\}.
\end{equation}
The fact that $H$ is continuous on a compact set means that it is uniformly continuous, therefore
\begin{equation}\label{4.71}
 \forall \, \epsilon >0\quad \exists \,\delta >0 \,\,\text{s.t.}\,\,|(t_1,\lambda_1)-(t_2,\lambda_2)|<\delta\,\,\Rightarrow \,\,|H(t_1,\lambda_1)-H(t_2,\lambda_2)|<\epsilon.
\end{equation}
Keeping Lemma \ref{3:prooflemma2} in mind, we set 
\begin{equation}\label{4.72}
 \lambda_2:=\max\left\{\lambda^*-\frac{\delta}{2},0\right\} \quad \text{and} \quad
 \epsilon :=\frac{1}{4}\inf_{t\in[0,1]} \epsilon_{H(t,\lambda^*)}>0,
\end{equation}
and choose $\lambda_1 \in \left(\lambda^*,\lambda^*+\frac{\delta}{2}\right)$ such that $H(1,\lambda_1)$ gives the opposite orientation to $H(1,0)$. This means we have $0\leqslant \lambda_2 < \lambda_1 \leqslant 1$ and $|\lambda_1-\lambda_2|< \delta$. Then using \eqref{4.71} we can easily deduce
\begin{equation}\label{4.73}
 |H(t,\lambda_1)-H(t,\lambda_2)|<\epsilon \quad \forall \,t\in[0,1].
\end{equation}
By the definition of $\epsilon$ we must also have that 
\begin{equation}\label{4.74}
 \mathcal{L}^d \left(U_{H(t,\lambda_1)}\cap U_{H(t,\lambda_2)}\right) >0\quad \forall \, t\in[0,1],
\end{equation}
so by our choice construction \eqref{4.68}, $H(t,\lambda_1)$ has the same orientation as $H(t,\lambda_2)$ for every $t\in[0,1]$. However, our assumption from \eqref{4.70} was that these two curves must have differing orientations at $t=1$, a contradiction. Therefore,
\begin{equation}\label{4.75}
 \left\{\, \mu\,|\,H(1,\mu)\,\,\text{has the opposite orientation to }H(1,0)=f(1)\,\right\}=\emptyset,
\end{equation}
meaning the choice of orientation on $U_y$ is path independent. To conclude we need to show that by this process we have defined a function $\vm$, which has $W^{1,2}$ regularity everywhere in the domain. Take any two points $x\neq y$, $x,y\in\Omega$, then if 
\begin{equation}\label{4.76}
 \mathcal{L}^d\left(U_x\cap U_y\right)=0 \,\,\Rightarrow \,\,\vm\in W^{1,2}\left(U_x\cup U_y,\mathbb{S}^{2}\right).
\end{equation}
Equally, by our construction, if $\mathcal{L}^d\left(U_x\cap U_y\right)>0$ then the orientations chosen in $U_x$ and $U_y$ agree on $U_x \cap U_y$. Thus in either case we have
\begin{equation}\label{4.77}
 \vm \in W^{1,2}\left(U_x\cup U_y,\mathbb{S}^{2}\right).
\end{equation}
Using this as a base case it is clear that a simple induction gives us
\begin{equation}\label{4.78}
 \vm \in W^{1,2}\left( \bigcup_{i=1}^N U_{x_i},\mathbb{S}^{2}\right),
\end{equation}
for any finite union. Now take any $U\subset\subset \Omega$; $\left\{\, U_x\,|\,x\in \overline{U}\,\right\}$ is an open cover for $\overline{U}$. This has a finite subcover $\left\{\, U_{x_i}\,\left|\,1\leq i \leq N,\,\,x_i\in \overline{U}\,\right\} \right.$. Then equation \eqref{4.78} tells us
\begin{equation}\label{4.79}
 \vm \in W^{1,2}\left( \bigcup_1^M U_{x_i},\mathbb{S}^{2}\right) \,\,\Rightarrow \,\,\vm \in W^{1,2}\left( U,\mathbb{S}^{2}\right)\,\,\Rightarrow \,\,\vm \in W^{1,2}_{\text{loc}}\left( \Omega,\mathbb{S}^{2}\right).
\end{equation}
In order to remove the local condition from \eqref{4.79} we note that $\vm\otimes\vm=\vn\otimes\vn \in W^{1,2}\left(\Omega,\mathbb{R}^{3\times3}\right)$ and 
\begin{equation}\label{4.80}
 \infty>\int_{\Omega}|\nabla \left( \vn\otimes\vn \right)|^2\,dx =\int_{\Omega}|\nabla \left( \vm\otimes\vm \right)|^2\,dx=\int_{\Omega}2|\nabla \vm|^2\,dx.
\end{equation}
Combining \eqref{4.80} with \eqref{4.79} we finally see that 
\begin{equation}\label{4.81}
 \vm \in W^{1,2}\left(\Omega,\mathbb{S}^{2}\right)
\end{equation}
\end{proof}

\begin{remark}
 This proposition shows that in a bounded simply connected domain in $\mathbb{R}^2$, or $\mathbb{R}^3$, uniaxial Landau-de Gennes theory with a constant scalar order parameter is directly equivalent to the Oseen-Frank theory.
\end{remark}

%
%
%
%
%
%

This lifting result is a nice extension of the result of Ball \& Zarnescu but we cannot restrict ourselves to just studying simply-connected domains for reasons which will become apparent in the equivalence theorems later in this section. Hence we will now prove the lifting result which is analogous to Proposition \ref{3:prop2 Ball-Zarnescu Extension} in a general domain. However in order to do so we will require the Whitney decomposition theorem.

\begin{theorem}\cite[Theorem 3, p.16]{stein1970singular}\label{whitney_thm}
Let $F\subset \mathbb{R}^n$ be a non-empty closed set and $\Omega:=F^c$. Then there is a set of closed cubes $\left\{ \, Q_j\, \subset \Omega\,|\,1\leqslant j\,\right\}$ such that 

\begin{itemize}
 \item $\Omega=\bigcup_1^\infty Q_j$
 \item $\text{int}( Q_j)\cap \text{int}\left( Q_k\right)= \emptyset$ if $j\neq k$
 \item $\text{diam}\left( Q_j\right) \leqslant d\left( Q_j,\partial \Omega\right) \leqslant 4 \text{diam}\left( Q_j\right)$
\end{itemize}

\end{theorem}

\begin{proposition}\label{3:prop3 Line field Vector field SBV}
 Suppose that $\Omega \subset \mathbb{R}^d$ is a bounded domain where $d=2,3$ and suppose that 
 \begin{equation}\label{4.82}
  \vn\otimes \vn \in W^{1,2}_{loc}\left( \Omega, \mathbb{R}^{3\times 3}\right)
 \end{equation}
 where $|\vn|=1$ almost everywhere in $\Omega$. Then there exists an 
 \begin{equation}\label{4.83}
  \vm \in SBV^2_{loc}\left( \Omega,\mathbb{S}^2\right) 
 \end{equation}
 such that $\vm\otimes \vm=\vn\otimes \vn$ almost everywhere, and $\vm_+=-\vm_-$ $\mathcal{H}^{d-1}$ almost everywhere on $S_\vm$.
 
 \end{proposition}
 
 \begin{proof}
  
  We immediately apply the decomposition theorem, Theorem \ref{whitney_thm}. This instantly gives us a set of cubes $\left.\left\{ \, Q_j\, \subset \Omega\,\right|\,1\leqslant j\,\right\}$ with the given properties. For ease of notation we will denote the interior of each of these closed cubes by $U_j$. We can certainly say that
  \begin{equation}\label{4.84}
   \vn\otimes \vn \in W^{1,2}\left( U_j,\mathbb{R}^{3\times 3} \right).
  \end{equation}
  The interior of each cube is clearly simply-connected so we can apply Theorem \ref{3:prop2 Ball-Zarnescu Extension} to find some $\vm \in W^{1,2}\left( U_j,\mathbb{S}^2\right)$ such that $\vm\otimes \vm =\vn\otimes \vn$ almost everywhere. We can do this in each of the cubes to find some 
  \begin{equation}\label{4.85}
   \vm\in W^{1,2}\left( \bigcup_1^\infty U_j,\mathbb{S}^2 \right).
  \end{equation}
  All we need to do to complete the proof is to investigate the behaviour of $\vm$ over the boundaries of these cubes to show that it is in SBV$_{loc}$. To do that we take some $V\subset\subset \Omega$ and begin by showing
  \begin{equation}\label{4.86}
   \left.\inf\left\{ \, \text{diam}Q_j\,\right|\,Q_j\cap V\neq \emptyset\,\right\} >0.
  \end{equation}
  We suppose for a contradiction that \eqref{4.86} doesn't hold. Then that means we can find a sequence of cubes $Q_{j_k}$ such that $\text{diam}\left( Q_{j_k}\right) \rightarrow 0$ and $Q_{j_k}\cap V \neq \emptyset$. However because $V$ is compactly contained in $\Omega$ 
  \begin{equation}\label{4.87}
   d(V,\partial \Omega)\geqslant \epsilon,
  \end{equation}
  for some $\epsilon>0$. The convergence of the diameters means that there is some $K$ such that for all $k\geqslant K$ $\text{diam}\left( Q_{j_k}\right) \leqslant \frac{\epsilon}{2}$, and therefore $d\left(Q_{j_k},\partial \Omega\right) \geqslant \frac{\epsilon}{2}$ as well. However we can see that this contradicts the third property of these cubes. Therefore \eqref{4.86} holds. Crucially this means that because our domain $\Omega$ is bounded, the set $V$ only intersects finitely many of these cubes and therefore 
  \begin{equation}\label{4.88}
   V\subset \bigcup_1^N Q_j.
  \end{equation}
 We define $U:=\bigcup_1^N \left( V\cap U_j\right) =V\setminus \left( \bigcup_1^N \partial Q_j\right)$, and we note that 
 \begin{equation}\label{4.89}
  \vm \in W^{1,2}\left( U,\mathbb{S}^2 \right)=W^{1,2}\left( V\setminus \left( \bigcup_1^N \partial Q_j\right),\mathbb{S}^2 \right).
 \end{equation}
 Equation \eqref{4.89} together with the fact that $\mathcal{H}^{d-1}\left( \bigcup_1^N \partial Q_j \right)<\infty$ means we can apply Ambrosio's result \cite[Proposition 4.4, p.213]{ambrosio2000functions} to deduce that 
 \begin{equation}\label{4.90}
  \vm \in SBV^2\left(V,\mathbb{S}^2\right) 
 \end{equation}
 with $\mathcal{H}^{d-1}\left( S_\vm\setminus  \bigcup_1^N \partial Q_j\right)=0$. Finally we just need to prove that $\vm_+=-\vm_-$ on the jump set $S_\vm$. Take some $x\in \partial Q_i \cap V$. We suppose that we can find an $\epsilon>0$ sufficiently small so that $B(x,\epsilon)$ only intersects two cubes, say $Q_i$ and $Q_j$. Note that this property is true for $\mathcal{H}^{d-1}$ almost every element of $\bigcup_1^N \partial Q_j$ because we are only discounting the edges and vertices of the cubes. We define the set $W:= B(x,\epsilon)\cap \bigcup_1^N \partial Q_j$, then by the Sobolev trace theorem we know that 
 \begin{equation}\label{4.91}
  \vm_-,\vm_+ \in W^{\frac{1}{2},2}\left( W,\mathbb{S}^2\right).
 \end{equation}
 If we define the sets 
 \begin{equation}\label{4.92}
  A:=\left\{\, x\in W\,|\,\vm_-=\vm_+\,\right\} \quad \text{and} \quad B:=\left\{\, x\in W\,|\,\vm_-=-\vm_+\,\right\},
 \end{equation}
then because $\vm_+ \otimes \vm_+ = \vm_- \otimes \vm_-$ $\mathcal{H}^{d-1}$ almost everywhere on $W$, clearly it must be the case that $\mathcal{H}^{d-1}\left( W\setminus\left( A\cup B\right)\right)=0$. Then we can use the integral representation for the fractional Sobolev space $W^{\frac{1}{2},2}$ to prove the conclusion. Equation \eqref{4.91} implies
\begin{equation}\label{4.93}
 \begin{split}
  \infty &> \int_W \int_W \frac{|\vm_-(x)-\vm_-(y)|^2}{|x-y|^d}\,dx\,dy \\
  &\geqslant \int_A \int_B \frac{|\vm_-(x)-\vm_-(y)|^2}{|x-y|^d}\,dx\,dy \\
  &= \int_A \int_B \frac{|\vm_+(x)+\vm_+(y)|^2}{|x-y|^d}\,dx\,dy.
 \end{split}
\end{equation}
However this can only be true if $\mathcal{H}^{d-1}(A)=0$ or $\mathcal{H}^{d-1}(B)=0$. Thus on $\bigcup_1^N Q_j$ we have either $\vm_-=\vm_+$ or $\vm_-=-\vm_+$ which means that on the jump set $S_\vm$ we must have that $\vm_+=-\vm_-$.

 \end{proof}
 
 This lifting result will be the crucial one when proving the equivalence between the uniaxial Landau-de Gennes theory and Ericksen's theory. We would however like to have an appropriate analogue to Corollary \ref{3: corollary 1} in the two and three dimensional case where we remove the locality assumptions from the previous result. Unfortunately we cannot do this because the Whitney decomposition theorem breaks up the domain into infinitely many simply connected pieces. The union of the boundaries of these components may have infinite $\mathcal{H}^{d-1}$ measure meaning that we are unable to apply \cite[Proposition 4.4, p.213]{ambrosio2000functions}. However we simply note here that it is possible to circumvent these difficulties if we have a Lipschitz domain but this is only a side note of this paper.
 
 \vspace{3 mm}

 %
 %
 %
 %
 %
 %
 
 Just as in Section \ref{sec:1D results}, now that we have established our lifting results we can use them to prove two equivalence results between the Landau-de Gennes theory and Ericksen's theory. However there is a very subtle but thorny technicality in multiple dimensions which needs to be addressed.
 
 \subsection*{Technicality}
 
 In the following result, Theorem \ref{3:2,3D Uniaxial Ericksen}, we will show the relationship between the uniaxial Landau-de Gennes theory and Ericksen's theory. By comparison with the one-dimensional result it would not be unreasonable to suppose that the set of admissible functions for the Ericksen problem should be 
 \begin{equation}\label{4.93.9}
  \mathcal{B}:=\left\{\,s\in \,W^{1,2}\left(\Omega\right),\,\,\vn \in SBV^2_{loc} \left( \Omega\setminus C_s,\mathbb{S}^2 \right) \, \left| 
\, \int_\Omega s^2 |\nabla \vn|^2 \,dx <\infty,\,\, \vn_+=-\vn_-\,\, {\rm on} \,\, S_\vn
 \,\right\}. \right.
 \end{equation}
  If we were to take a pair $(s,\vn) \in \mathcal{B}$, in order to prove the equivalence of the admissible functions, we would need to show that $s\vn \otimes \vn \in W^{1,2}\left( \Omega ,\mathbb{R}^{3 \times 3} \right)$. Since $SBV^2_{loc}\cap L^\infty$ is closed under products we would find $\vn \otimes \vn \in SBV^2_{loc} \left(\Omega \setminus C_s,\mathbb{R}^{3\times 3} \right)$. The fact that $\vn_+ = -\vn_-$ on $S_\vm$ would then indicate that $\vn \otimes \vn \in W^{1,2}_{loc}\left( \Omega \setminus C_s,\mathbb{R}^{3 \times 3} \right)$. Finally the joint regularity would imply that 
  \begin{equation}\label{4.93.10}
   s\vn \otimes \vn \in W^{1,2}\left( \Omega \setminus C_s, \mathbb{R}^{3\times 3} \right).
  \end{equation}
  From an intuitive point of view it is only a short step to deduce that \eqref{4.93.10} implies $s\vn \otimes \vn \in W^{1,2}\left( \Omega ,\mathbb{R}^{3 \times 3} \right)$. However lots of subtle points about the fine regularity of Sobolev functions come into play. We would need Lemma \ref{3:lemma3} to hold without the additional continuity assumption. This would be true if we could show, for example, that for $f \in W^{1,p}\left( \Omega \right)$
  \begin{equation}\label{4.93.11}
   \mathcal{H}^{d-1} \left( C_f \setminus \left\{ f = 0\right\} \right) =0.
  \end{equation}
 We conjecture that \eqref{4.93.11} always holds because Sobolev functions should not have a discontinuity set with Hausdorff dimension greater than or equal to $d-1$. However presently we are unable to prove this claim. Therefore in the multi-dimensional results below we will simply assume that $s(\vn \otimes \vn) \in W^{1,2}\left( \Omega , \mathbb{R}^{3\times 3} \right)$ as our joint regularity condition. The reasoning above shows that this is not an oversimplification; it is only used to deal with this small technicality which we believe to be true in any case.
 
 \subsection{Uniaxial Landau-de Gennes}
 
  \begin{theorem}\label{3:2,3D Uniaxial Ericksen}
  Let $\Omega\subset \mathbb{R}^d$ be a bounded domain with $d=2,3$. Then minimising $I_L(\vQ)$ over
 \begin{equation}\label{4.94}
  \mathcal{A}:=\left\{\, \vQ \in W^{1,2}\left( \Omega,\mathbb{R}^{3\times 3} \right)\,\left|\,\vQ^T = \vQ,\,\, {\rm Tr}\,\vQ = 0,\,\, \vQ\, \text{uniaxial}\,\right\} \right.
 \end{equation}
 is equivalent to minimising $I_E(s,\vn)$ over
 \begin{equation}\label{4.95}
  \mathcal{B}:=\left\{\,s\in \,W^{1,2}\left(\Omega\right),\,\,\vn \in SBV^2_{loc} \left( \Omega\setminus C_s,\mathbb{S}^2 \right) \, \left| 
\, s(\vn \otimes \vn) \in W^{1,2}\left(\Omega ,\mathbb{R}^{3\times 3} \right)
 \,\right\}. \right.
 \end{equation}
 
 \end{theorem}
 
 \begin{proof}
  
  This proof closely follows that of Theorem \ref{3:1D-Uniaxial-Ericksen} so we will skim over the details of this proof which have already been addressed there. For ease of notation we denote $\Omega':=\Omega\setminus C_s$. We take a $\vQ \in \mathcal{A}$ then we know that we can write $\vQ$ as $s\left( \vn \otimes \vn-\frac{1}{3}Id \right)$. We can immediately use the logic from \eqref{4.38} and \eqref{4.39} to deduce that
\begin{equation}\label{4.96}
s\in W^{1,2}\left( \Omega\right).
\end{equation}
Similarly we can use the same reasoning from the proof of Theorem \ref{3:1D-Uniaxial-Ericksen} again to deduce the regularity of the line field. Using Lemma \ref{3:lemma6} we find
\begin{equation}\label{4.97}
\vn\otimes\vn \in W^{1,2}_{loc} \left( \Omega',\mathbb{R}^{3\times 3} \right).
\end{equation}
We are now in a position to use our previous lifting result to convert this line-field into a vector-field. Applying Proposition \ref{3:prop3 Line field Vector field SBV} to $\vn\otimes\vn$ tells us that there exists some
\begin{equation}\label{4.98}
\vm \in SBV^2_{loc}\left( \Omega',\mathbb{S}^2 \right)\,\, \text{such that}\,\, \vm\otimes \vm =\vn \otimes \vn \,\,a.e.
\end{equation}
and $\vm_+=-\vm_-$ $\mathcal{H}^{d-1}$ almost everywhere on the jump set $S_\vm$. This completes the inclusion as we have shown that $(s,\vm) \in \mathcal{B}$. The reverse inclusion is trivial, we take a pair $(s,\vn) \in \mathcal{B}$ and then it is clear from their definitions that 
\begin{equation}\label{4.99}
 Q:=s\left( \vn \otimes \vn - \frac{1}{3}I \right) \in W^{1,2} \left( \Omega ,\mathbb{R}^{3 \times 3} \right).
\end{equation}
This establishes the link between the sets of admissible functions and the formal equivalence of the Lagrangians is exactly the same as in Theorem \ref{3:1D-Uniaxial-Ericksen}. 

\end{proof}

%
%
%
%
%
%

 \subsection{Biaxial Landau-de Gennes}
 
 \begin{theorem}\label{3:2,3D Biaxial Ericksen}
  Let $\Omega \subset \mathbb{R}^d$ be a bounded domain with $d=2,3$. Then minimising $I_L(\vQ)$ over 
  \begin{equation}\label{4.104}
   \mathcal{A}:=\left\{ \vQ \in W^{1,2}\left( \Omega,\mathbb{R}^{3\times3}\right)\, \left| \, \vQ^T = \vQ,\,\, {\rm Tr}\, \vQ=0\, \right.\right\}.
  \end{equation}
  is equivalent to minimising $I_{BE}(\vn,\vm,s_1,s_2)$ over 
  \begin{equation}\label{4.105}
     \mathcal{B}:=\left\{\,
     \begin{array}{c}s_1,-s_2\in W^{1,2}\left(\Omega,[0,\infty)\right),\,\, \\
     \vn \in SBV^2_{loc} \left( \Omega\setminus C_{s_1},\mathbb{S}^2 \right),\,\, \vm \in SBV^2_{loc} \left( \Omega\setminus C_{s_2},\mathbb{S}^2 \right)          
     \end{array}
 \, \left| 
\, \begin{array}{c}s_1 (\vn \otimes \vn) \in W^{1,2}\left( \Omega ,\mathbb{R}^{3\times 3} \right) \\
    s_2 (\vm \otimes \vm) \in W^{1,2}\left( \Omega ,\mathbb{R}^{3\times 3} \right)
   \end{array}
 \right\}. \right.
  \end{equation}

 \end{theorem}
 
 \begin{proof}
  
  The proof of this result follows similar lines to that of Theorem \ref{3:1D-Biaxial-Ericksen}. We begin by taking $\vQ \in \mathcal{A}$. The variable $\vQ$ can be written as
  \begin{equation}\label{4.106}
   \vQ:=s_1\left( \vn\otimes\vn-\frac{1}{3}I\right) - s_2\left( \vm\otimes\vm-\frac{1}{3}I\right),
  \end{equation}
  where $s_1\geqslant 0$, $s_2\leqslant 0$ and $|\vn|=|\vm|=1$. Using the logic from the proof of Theorem \ref{3:1D-Biaxial-Ericksen}, straight away we deduce
  \begin{equation}\label{4.107}
   s_1,s_2 \in W^{1,2}\left(\Omega\right).
  \end{equation}
  With a little consideration we see that we can write the conclusion of Lemma \ref{3:lemma5} in a slightly different form.
  \begin{equation}\label{4.107.1}
   \begin{split}
    |\vQ (x)-\vQ (y) |^2 \geqslant &
    \frac{1}{6} \left( (s_1(x)-s_1(y))^2 + (s_2(x)-s_2(y))^2 \right)+ \frac{5}{6}s_1(x)s_1(y)\left| \vn \otimes \vn(x)- \vn\otimes \vn(y) \right|^2\\ &
    +\frac{5}{6}s_2(x)s_2(y)\left| \vm \otimes \vm(x)- \vm\otimes \vm(y) \right|^2 \\ &
    +\frac{1}{6}\left|s_1 \vn \otimes \vn (x) - s_1 \vn \otimes \vn(y) \right|^2+\frac{1}{6}\left|s_2 \vm \otimes \vm (x) - s_2 \vm \otimes \vm(y) \right|^2
   \end{split}
  \end{equation}
  Therefore by applying our difference quotient logic to the final two terms in the equation above we deduce that
  \begin{equation}\label{4.108}
   s_1\vn\otimes \vn\in W^{1,2}\left( \Omega,\mathbb{R}^{3\times3}\right)\,\, \text{and}\,\,  s_2\vm\otimes \vm\in W^{1,2}\left( \Omega,\mathbb{R}^{3\times3}\right).
  \end{equation}
  Then we can apply both Lemma \ref{3:lemma6} and Proposition \ref{3:prop3 Line field Vector field SBV} to each of these functions, and find
  \begin{equation}\label{4.109}
  \vu\in SBV^2_{loc}\left( \Omega\setminus C_{s_1},\mathbb{S}^{2}\right)\,\, \text{and}\,\,  \vv\in SBV^2_{loc}\left( \Omega\setminus C_{s_2},\mathbb{S}^{2}\right).
  \end{equation}
  where $\vu\otimes\vu=\vn\otimes\vn$ and $\vv\otimes\vv = \vm\otimes\vm$ almost everywhere. This concludes the first inclusion. For the reverse one we take some $(\vn,\vm,s_1,s_2)\in \mathcal{B}$ then using joint regularity assumptions that
  \begin{equation}\label{4.111}
   s_1\vn\otimes\vn,s_2\vm\otimes\vm \in W^{1,2}\left( \Omega,\mathbb{R}^{3\times 3} \right),
  \end{equation}
  we clearly have 
  \begin{equation}\label{4.112}
   \vQ:= s_1\left( \vn\otimes\vn -\frac{1}{3}I\right) - s_2\left( \vm\otimes\vm-\frac{1}{3}I\right)\in W^{1,2}\left( \Omega,\mathbb{R}^{3\times 3}\right).
  \end{equation}
  This finalises the one-to-one correspondence of $\mathcal{A}$ and $\mathcal{B}$. The formal equivalence of the Lagrangians was already established in the proof of Theorem \ref{3:1D-Biaxial-Ericksen} so the assertion is complete.
  
 \end{proof}
 
 %
 %
 %
 %
 %
 %
 
 \section{A Different Approach}\label{section: new approach}
 
 These theorems provide an interesting insight into the precise function space equivalences between the Landau-de Gennes, Ericksen and Oseen-Frank theories. However as we see from the statement of Theorem \ref{3:2,3D Uniaxial Ericksen} for example, it would be a very great analytical exercise to find minimisers of the Ericksen energy given that the director $\vn$ and the scalar order parameter $s$ are not independent. The biaxial energy \eqref{4.6} is yet another order of magnitude more difficult if one were looking for any sort of explicit minimiser. If, in isolation, we were to try and minimise
 \begin{equation}\label{4.113}
  I_E(\vn,s)=\int_\Omega 2s^2\left( |\nabla\vn|^2+2t\vn\cdot\nabla \times \vn+t^2\right) +\frac{2}{3}|\nabla s|^2+\sigma(s)\,dx.
  \end{equation}
over
\begin{equation}\label{4.114}
 \mathcal{B}:=\left\{\,s\in \,W^{1,2}\left(\Omega\right),\,\,\vn \in SBV^2_{loc} \left( \Omega\setminus C_s,\mathbb{S}^2 \right) \, \left| 
\, s\vn \otimes \vn \in W^{1,2}\left( \Omega ,\mathbb{R}^{3\times3} \right)\,
 \right\}. \right.
\end{equation}
it would not even be clear whether we could show the existence of a global minimiser through the direct method of the calculus of variations. In fact, Theorem \ref{3:2,3D Uniaxial Ericksen} ensures that a global minimiser does exist. Thus in this section we will use the above results as motivation to justify a new liquid crystal continuum model which is tractable for an analytical approach. For example, with the scalar order parameter we are not so much interested in its exact value, but where it vanishes, as this represents a change in phase of the liquid crystal sample. Therefore we are principally interested in the two different states of $s$: where it is zero, and where it is non-zero. Hence we will use a $\Gamma$-convergence result to retain the phase change information whilst simplifying the energy greatly. We recall the definition of $\Gamma$-convergence.

\begin{definition}[$\Gamma$-convergence]\label{def: gamma convergence}
 Let $X$ be a metric space. Then $f_j:X\rightarrow \mathbb{R}\cup\left\{ \pm \infty\right\}$ $\Gamma$-converges to $f:X\rightarrow \mathbb{R}\cup\left\{ \pm \infty\right\}$, if for every $x\in X$ we have the following two properties.
 \begin{enumerate}
  \item If $x_j \rightarrow x$ in $X$ then $f(x)\leqslant \liminf_j f_j(x_j)$.
  
  \item There exists some sequence $x_j \rightarrow x$ in $X$ such that $f(x)\geqslant \limsup_j f_j(x_j)$.

 \end{enumerate}
\end{definition}

 %
 %
 %
 %
 %
 %

\subsection{$\Gamma$-Convergence Result}

Suppose that $d\in \left\{1,2,3\right\}$ and $\Omega\subset \mathbb{R}^d$ is an open set, with Lipschitz boundary. Suppose that $\sigma: \mathbb{R}\rightarrow [0,\infty)$ is a continuous function with exactly two distinct zeroes at $s_-$ and $s_+$. Let $\epsilon>0$ and $F_\epsilon, F:L^1\left(\Omega,\mathbb{R}\right)\rightarrow \overline{\mathbb{R}}$ be defined by 
\begin{equation}\label{4.117}
 F_\epsilon (v):=\left\{ \begin{array}{lcc} \int_\Omega \frac{2\epsilon}{3}|\nabla v|^2+\frac{1}{\epsilon}\sigma(v)\,dx & & v\in W^{1,2}(\Omega) \\
 \infty & & \text{otherwise}
\end{array} \right.
\end{equation}
and
\begin{equation}\label{4.118}
 F(v):=\left\{ \begin{array}{lcc} C\mathcal{H}^{d-1}\left( S_v\right) &  & v\in SBV^2\left( \Omega, \left\{ s_-,s_+\right\}\right) \\ \infty & & \text{otherwise} \end{array} \right.
\end{equation}
where $C=2\sqrt{\frac{2}{3}}\int_{s_-}^{s_+} \sqrt{\sigma(\lambda)}\, d\lambda$.

\begin{theorem}\label{gamma convergence}
 The functional $F_\epsilon$ $\Gamma$-converges to $F$ with the strong topology on $L^1(\Omega,\mathbb{R})$. Furthermore if $u_\epsilon$ is a minimiser of $F_\epsilon$ and  $u_\epsilon \rightarrow u$ as $\epsilon \rightarrow 0$ in $L^1(\Omega)$ then $u$ is a minimiser of $F$.
\end{theorem}

\begin{proof}
 
 The proof has three main steps. We show the two inequalities required for $\Gamma$-convergence and then address the convergence of minimisers statement. This result is similar to one proved by Modica \cite[Theorem 1]{modica1987gradient} and therefore we will utilise some results from his paper to shorten our proof. 
 
 \vspace{3 mm}
 
 We begin by showing the lower bound for $\Gamma$-convergence. Let $\left( v_\epsilon \right)$ be a sequence of $L^1(\Omega)$ functions such that $v_\epsilon \rightarrow v \in L^1(\Omega)$ as $\epsilon \rightarrow 0$. We need to show that 
 \begin{equation}\label{4.119}
  \liminf_{\epsilon \rightarrow 0} F_\epsilon (v_\epsilon) \geqslant F(v).
 \end{equation}
 Without loss of generality we can assume that $\liminf_{\epsilon \rightarrow 0} F_\epsilon(v_\epsilon) = \lim_{\epsilon \rightarrow 0} F_\epsilon (v_\epsilon)<\infty$. As $v_\epsilon \rightarrow v$ in $L^1(\Omega)$, up to a subsequence we may also assume that $v_\epsilon \rightarrow v$ almost everywhere. Then the continuity of $\sigma$, together with Fatou's Lemma implies 
 \begin{equation}\label{4.120}
  \int_\Omega \sigma(v)\,dx \leqslant \liminf_{\epsilon \rightarrow 0} \int_\Omega \sigma(v_\epsilon)\,dx \leqslant \liminf_{\epsilon \rightarrow 0}C\epsilon = 0. 
 \end{equation}
 Therefore $v(x)=s_-$ or $s_+$ almost everywhere. As the function $\overline{v_\epsilon}=s_- \vee v_\epsilon \wedge s_+ \in H^1(\Omega)$, without loss of generality we can assume that $v_\epsilon(x)\in [s_-,s_+]$ since $F_\epsilon(\overline{v_\epsilon})\leqslant F_\epsilon(v_\epsilon)$. Using Young's inequality and H\"{o}lder's inequality we find
 \begin{equation}\label{4.121}
 \begin{split}
  F_\epsilon(v_\epsilon) & \geqslant 2\left( \int_\Omega \frac{2}{3}|\nabla v_\epsilon|^2\,dx \right)^\frac{1}{2}\left( \int_\Omega \sigma(v_\epsilon)\,dx\right)^\frac{1}{2} \\
  & \geqslant 2\sqrt{\frac{2}{3}}\int_\Omega |\nabla v_\epsilon|\sqrt{\sigma(v_\epsilon)}\,dx.
 \end{split}
 \end{equation}
 If we introduce the function $\psi(t)=\int_{s_-}^t \sqrt{\sigma(\lambda)}\, d\lambda$, then as it is continuous, $v_\epsilon \rightarrow v$ in $L^1(\Omega)$, and $s_-\leqslant v_\epsilon \leqslant s_+$, we deduce from the dominated convergence theorem that $\psi(v_\epsilon)\rightarrow \psi(v)$ in $L^1(\Omega)$. We also know that
 \begin{equation}\label{4.122}
 \infty > \liminf_{\epsilon \rightarrow 0} F_\epsilon (v_\epsilon) \geqslant 2\sqrt{\frac{2}{3}} \liminf_{\epsilon \rightarrow 0} \int_\Omega |\nabla \left( \psi(v_\epsilon) \right) |\,dx.
 \end{equation}
 These deductions allow us infer that $\psi(v) \in BV(\Omega)$ \cite[Proposition 10.1.1]{buttazzo2006variational} and 
 \begin{equation}\label{4.123}
  \liminf_{\epsilon \rightarrow 0} F_\epsilon(v_\epsilon) \geqslant  2\sqrt{\frac{2}{3}} \int_\Omega |\nabla \left( \psi(v)\right) |\, dx.
 \end{equation}
 However $v$ only takes two possible values. Therefore we can find the exact form of \eqref{4.123} and conclude 
 \begin{equation}\label{4.124}
  \liminf_{\epsilon \rightarrow 0} F_\epsilon (v_\epsilon) \geqslant \left( 2\sqrt{\frac{2}{3}}\int_{s_-}^{s_+} \sqrt{\sigma(\lambda)}\,d\lambda \right) \mathcal{H}^{d-1}(S_v).
 \end{equation}
 For the upper bound in the definition of $\Gamma$-convergence we take some $v\in L^1(\Omega)$ and we need a sequence $v_\epsilon \rightarrow v$ in $L^1(\Omega)$ such that 
 \begin{equation}\label{4.125}
  \limsup_{\epsilon \rightarrow 0} F_\epsilon(v_\epsilon)=\lim_{\epsilon \rightarrow 0} F_\epsilon (v_\epsilon)=F(v).
 \end{equation}
 Without loss of generality we can assume $v \in SBV^2(\Omega,\left\{ s_-,s_+ \right\})$. As $v$ only takes two values we can write it as $s_-+(s_+-s_-)\chi_{\left\{v=s_+\right\}}$. Following Modica's argument, initially we assume a smoothness property, then prove \eqref{4.125} in general by approximation. Suppose there exists some open set $A\subset \mathbb{R}^d$ such that $\partial A$ is non-empty, compact, smooth, $\mathcal{H}^{d-1}(\partial A \cap \partial \Omega)=0$ and $v=s_-+(s_+-s_-)\chi_A$. Then we can find a set of Lipschitz functions $v_\epsilon \rightarrow v$ in $L^1(\Omega)$ which satisfy \eqref{4.125} \cite[Proposition 2]{modica1987gradient}. Now we take a general $v \in SBV^2(\Omega,\left\{ s_-,s_+ \right\})$ and \cite[Lemma 1]{modica1987gradient} implies that we can find sets $(A_j)$ satisfying the above regularity conditions such that 
 \begin{equation}\label{4.126}
  \mathcal{H}^{d-1}( A_j ) \rightarrow \mathcal{H}^{d-1} \left( S_v \right).
 \end{equation}
 For ease we can assume that $\mathcal{H}^{d-1}(A_j)\leqslant \mathcal{H}^{d-1}(S_v)+\frac{1}{j}$ for every $j$. If we define $v^j :=s_-+(s_+-s_-)\chi_{A_j}$, then by our initial step, we can find Lipschitz functions $(v_\epsilon^j)$ such that
 \begin{equation}\label{4.127}
  \limsup_{\epsilon \rightarrow 0} F_\epsilon(v^j_\epsilon ) \leqslant C \mathcal{H}^{d-1}(A_j) \leqslant C\left(\mathcal{H}^{d-1}(S_v)+\frac{1}{j}\right).
 \end{equation}
 Using a simple diagonalisation argument we can find some sequence $\epsilon(j) \rightarrow 0$ as $j\rightarrow \infty$ such that 
 \begin{equation}\label{4.128}
  \limsup_{j \rightarrow \infty} F_{\epsilon(j)} \left(v^j_{\epsilon(j)} \right) \leqslant C \mathcal{H}^{d-1} (S_v).
 \end{equation}
 Therefore $F_\epsilon$ does $\Gamma$-converge to $F$ in the appropriate sense. So to finish the proof we just need to show the standard $\Gamma$-convergence property: convergence of minimisers. Suppose that we have a sequence of minimisers $(u_\epsilon)$ such that $u_\epsilon \rightarrow u$ in $L^1(\Omega)$. We first show that 
 \begin{equation}\label{4.129}
  F(u)\leqslant F(v)
 \end{equation}
 for every $v\in SBV^2 \left(\Omega,\left\{ s_-,s_+ \right\} \right)$ which satisfies the above smoothness property. However this is clear because if $(v_\epsilon)$ is the series of Lipschitz functions approximating $v$ in $L^1(\Omega)$ then 
 \begin{equation}\label{4.130}
  F(u)\leqslant \liminf_{\epsilon \rightarrow 0} F_\epsilon (u_\epsilon) \leqslant \liminf_{\epsilon \rightarrow 0} F_\epsilon (v_\epsilon) \leqslant \limsup_{\epsilon \rightarrow 0} F_\epsilon (v_\epsilon) \leqslant F(v).
 \end{equation}
 Now if we take an arbitrary $v\in SBV^2\left( \Omega, \left\{ s_-,s_+ \right\} \right)$, we can use a very similar approximation argument as above to deduce 
 \begin{equation}\label{4.131}
  F(u)\leqslant F(v).
 \end{equation}
 Thus $u$ is indeed a minimiser of $F$.
 
\end{proof}

In order to be able to use Theorem \ref{gamma convergence} as a suitable approximation we need to employ some knowledge about the bulk energy $\sigma$. We know that it has the form
\begin{equation}\label{4.132}
\sigma(s):=\frac{a}{2}s^2+\frac{b}{3}s^3+\frac{c}{4}s^4
\end{equation}
where $c>0$, and $a$ is temperature dependent. It is often assumed that $a=\alpha(T-T^*)$ where $\alpha>0$ and $T^*$ is the temperature at which the isotropic state becomes unstable. If we follow this convention on the form of $a$ then there is a critical temperature, $T_c>T^*$ (following the notation of Virga \cite{virga1994variational}), such that $\sigma$ has exactly two global minimisers: $0$ and $s_+$. By the addition of a constant, without loss of generality we can assume that $\sigma(0)=\sigma(s_+)=0$ at or near this critical temperature.

\vspace{3mm}

Finally, and most importantly, we need to consider the size of the bulk constant relative to the elastic constants to ensure that Theorem \ref{gamma convergence} is a physically appropriate limit to take. A typical set of values for the Frank elastic constants \cite{stewart2004static} is
\begin{equation}\label{4.133}
 K_1 = 4.5 \times 10^{-12} N,\quad K_2 = 2.9\times 10^{-12}N,\quad K_3 = 9.5\times 10^{-12}N.
\end{equation}
Of course exact values vary between liquid crystals but they are certainly of the order of $10^{-11}N$ or less. On the other hand the bulk constants $\alpha$, $b$, and $c$ are very much larger. They have been measured fewer times but a set of values was experimentally found to be \cite{wojtowicz1975introduction}
\begin{equation}\label{4.134}
 \alpha = 0.042 \times 10^6 Nm^{-2}K^{-1},\quad b=-0.64 \times 10^6 Nm^{-2},\quad c = 0.35 \times 10^6 Nm^{-2}
\end{equation}
Throughout this paper we have been tacitly using the one constant approximation for our elastic energies with an elastic constant of the order $10^{-11}N$. However, as can be seen from \eqref{4.113} we have been considering energies where we have divided through by this Frank constant. Therefore the value of the bulk constants in our functional \eqref{4.113} will be of the order $10^{15}$ or $10^{16}$. This means that it is a justified idea to use the $\Gamma$-convergence result to provide an approximation to our variational problem in order to simplify it. Theorem \ref{gamma convergence} therefore tells us that our variational problem can be approximated by the minimisation of
\begin{equation}\label{4.135}
 I_E(\vn,s)=\int_\Omega 2s^2\left( |\nabla \vn|^2 + 2t\vn\cdot\nabla\times\vn + t^2\right) \, dx + K \mathcal{H}^{d-1}\left( S_s\right)
\end{equation}
over 
\begin{equation}\label{4.136}
 \mathcal{B}:=\left\{ \,  s\in SBV^2 \left(\Omega, \left\{ 0,s_+\right\} \right), \vn \in SBV^2 \left(\Omega\setminus C_s, \mathbb{S}^2 \right)\, \left| \,  
 \vn_-=-\vn_+\,\,\mathcal{H}^{d-1}\,\,a.e.\,\,\text{on}\,\,S_\vn \,
    \right.\right\}.
\end{equation}
This is simpler but we still have the issue of two variables which are not independent from each other. So now if we consider the variable $\vm:=s\vn$ then because we know that 
\begin{equation}\label{4.137}
 s = s_+\chi_{\left\{ s=s_+ \right\}} \in SBV^2(\Omega), \,\, \text{and}\,\, \vn \chi_{\left\{ s=s_+ \right\}} \in SBV^2 \left( \left\{ s=s_+ \right\}, \mathbb{S}^2 \right)
\end{equation}
and $SBV\cap L^\infty$ is an algebra of functions, we find 
\begin{equation}\label{4.138}
 \vm \in SBV^2 \left(\Omega, s_+\mathbb{S}^2 \cup \left\{ 0 \right\} \right).
\end{equation}
Therefore if we now consider the free-energy \eqref{4.135} to be just a function of $\vm$ we arrive at the form
\begin{equation}\label{4.139}
 I(\vm):= \int_{\left\{\vm\neq 0\right\}} |\nabla \vm|^2+ 2t\vm\cdot\nabla\times \vm+ t^2|\vm|^2\,dx+K \mathcal{H}^{d-1}\left( \partial \left\{ \vm = 0 \right\} \right).
\end{equation}
We now find ourselves with a free-energy which is strongly related to the Oseen-Frank free energy. The set $\left\{ \vm \neq 0 \right\}$ corresponds to the liquid crystal phase of the sample, whereas $\left\{ \vm = 0 \right\}$ represents wherever the molecules are isotropic. Even though this appears to be an Oseen-Frank energy, because it was derived from the uniaxial Landau-de Gennes model it has some immediate advantages. Firstly, all defects have finite energy with this model. 

\vspace{3 mm}

\begin{figure}[ht]
\centering
\subfloat{\label{fig:Defect 1}\includegraphics[scale=0.2]{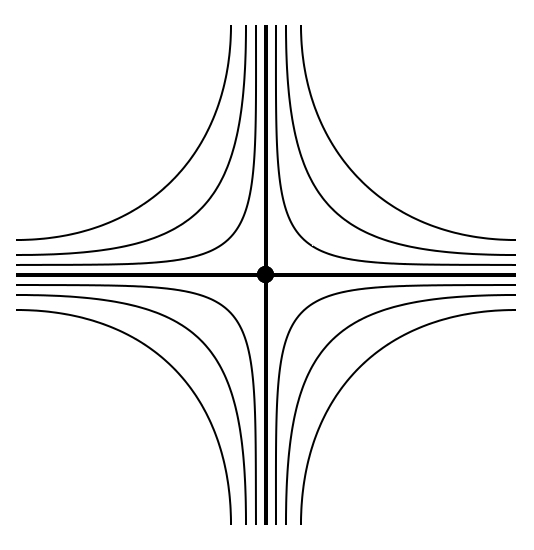}} \quad
\subfloat{\label{fig:Defect 2}\includegraphics[scale=0.132]{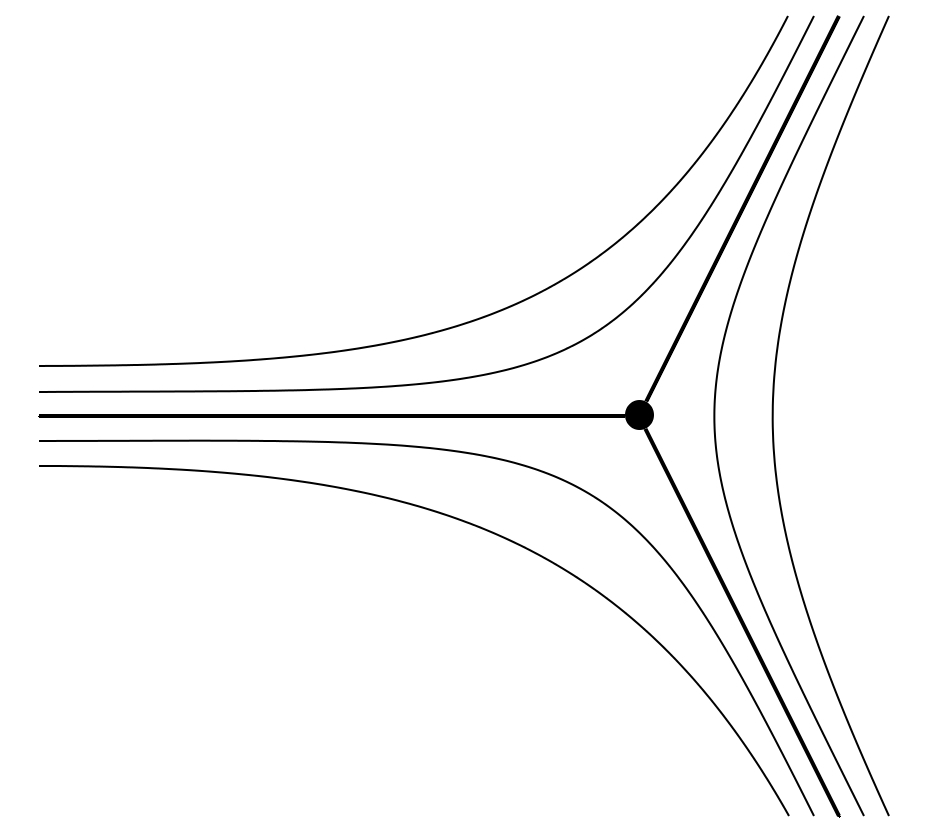}}
\caption{Two defect profiles}
\label{fig:defects}
\end{figure}

The defects shown in Figure \ref{fig:defects} can be given a finite energy in this model by inserting a small core of isotropic region in the center. Secondly, weak anchoring is automatically included in the model because with variational problems set in $BV$, it cannot be guaranteed that a minimiser satisfies imposed boundary conditions. Unfortunately, the variational problem as given by \eqref{4.138} and \eqref{4.139} is not well-posed in the sense that, without a control on the size jump set of the unit vector part $\vn$, a global minimiser may not exist. This brings us to our conclusion, we propose minimising
\begin{equation}\label{4.140}
\begin{split}
I(\vm) :=&K_E \int_{\left\{\vm\neq 0\right\}} |\nabla \vm|^2+ 2t\vm\cdot\nabla\times \vm+ t^2|\vm|^2\,dx+K \mathcal{H}^{d-1}\left( S_\vm \right) \\ 
=&K_E \int_\Omega |\nabla \vm|^2+ 2t\vm\cdot\nabla\times \vm+ t^2|\vm|^2\,dx+K \mathcal{H}^{d-1}\left( S_\vm \right)
\end{split}
\end{equation}
over $\mathcal{B}$ as given by \eqref{4.138}. After rescaling $\vm$, and the functional itself, this is equivalent to minimising 
\begin{equation}\label{4.141}
I(\vn):= \int_\Omega |\nabla \vn|^2+ 2t\vn\cdot\nabla\times \vn+ t^2|\vn|^2\,dx+K \mathcal{H}^{d-1}\left( S_\vn \right).
\end{equation}
over
\begin{equation}\label{4.142}
 \vn \in SBV^2 \left(\Omega, \mathbb{S}^2 \cup \left\{ 0 \right\} \right).
\end{equation}

%
%
%
%
%
%
%

\section{Bounded Variation Variational Problems}\label{section: SBV}

%
%
%
%
%
%

\subsection{The Problem}

There are a number of small technical issues which need to be considered when setting up variational problems with functions of bounded variation. For the remainder of the paper, unless stated otherwise, our domain $\Omega \subset \mathbb{R}^d$ will be a bounded Lipschitz domain. We will of course want to apply boundary conditions to our problem, but this creates an issue. Using the direct method of the calculus of variations, we will be able to extract a subsequence from the minimising sequence $\left( \vn_j \right)$ which converges weakly$^*$ in SBV. However traces are not preserved under weak$^*$ convergence. Therefore we will actually consider our variational problem over an enlarged domain $\Omega_\epsilon$ in the following sense. We look to minimise the integral functional
\begin{equation}\label{5.7}
 \mathscr{F}({\bf n})=\int_{\Omega_\epsilon} w({\bf n},\nabla{\bf n})\,dx+K\mathcal{H}^2(S_{{\bf n}}),
\end{equation}
over the set of admissible functions 
\begin{equation}\label{5.8}
 \mathcal{A}:=\left\{\,\left. \vn\in SBV^2\left(\Omega_{\epsilon},\mathbb{S}^2 \cup \left\{ 0\right\} \right)\,\right|\,\vn={\bf g}\,\,\text{on}\,\,\Omega_{\epsilon}\backslash\overline{\Omega}\,\right\},
\end{equation}
where $\Omega_{\epsilon}$ is an open set containing $\Omega$ and ${\bf g} \in W^{1,2}\left( \Omega_\epsilon \setminus \Omega, \mathbb{S}^2 \right)$. As we shall see, this extended domain idea allows us to prove the existence of a minimiser for our problem while at the same time incorporating weak boundary conditions since $S_\vn$ is allowed to intersect $\partial \Omega$. After this section, when there is no room for confusion we will drop the subscript on the extended domain $\Omega_\epsilon$. However we maintain the understanding that we are studying an extended domain and the jump set can lie on the portion of the boundary $\partial \Omega$ where boundary conditions are applied. In Section \ref{section: cholesteric problem} we will also study a problem involving periodic boundary conditions. For completeness we note that these boundary conditions can be incorporated into the problem using a similar logic to the Dirichlet conditions above, and lead to the same existence result and Euler-Lagrange equations. In terms of liquid crystal theory our Lagrangian will be the familiar one-constant Oseen-Frank free energy
\begin{equation}\label{5.9}
 w({\bf n},\nabla{\bf n})= \left[ |\nabla \vn|^2 + 2t\vn\cdot \nabla \times \vn + t^2|\vn|^2 \right].
\end{equation}
The only critical property of the free energy functional that we require for the existence proof is that
\begin{equation}\label{5.10}
 w({\bf n},\nabla{\bf n})\geqslant \phi(|\nabla {\bf n}|),
\end{equation}
where $\phi(z)$ is a convex function in $z$ such that 
\begin{equation}\label{5.11}
 \lim_{t\rightarrow\infty}\frac{\phi(z)}{z}=\infty.
\end{equation}

%
%
%
%
%
%

\subsection{Existence of a Minimiser}

\begin{theorem}\label{4: existence of minimiser}
Consider the minimisation problem outlined in \eqref{5.7} and \eqref{5.8}. Then there exists an ${\bf n}\in \mathcal{A}$ such that 
\begin{equation}\label{5.12}
 \mathscr{F}({\bf n})=\inf_{\vm\in \mathcal{A}} \mathscr{F}(\vm).
\end{equation}
\end{theorem}

\begin{proof}

This result is an adaptation of \cite[Thm 5.24]{ambrosio2000functions} to our specific setup. We have that $\inf_{\vm\in \mathcal{A}} \mathscr{F}(\vm)>-\infty$, so by taking a minimising sequence ${\bf n}^j$, the structure of the functional itself implies
\begin{equation}\label{5.14}
 \int_\Omega \phi\left( |\nabla \vn^j| \right) \,dx + K\mathcal{H}^{d-1}\left( S_{\vn^j} \right)\leqslant M
\end{equation}
for every $j$. This means we can apply Theorem \ref{4: SBV compactness} to deduce that there is a subsequence (which we relabel) such that 
\begin{equation}\label{5.15}
 {\bf n}^j \tostar {\bf n} \quad \text{in}\,\,BV,\,\,\text{where}\,\,\vn\in SBV(\Omega,\mathbb{R}^3).
\end{equation}
Furthermore, since $\mathscr{F}(\vn) < \infty$ we infer that $\vn \in SBV^2(\Omega,\mathbb{R}^3)$. Now we can simply apply two results from the literature. The first is that under the assumption \eqref{5.14}, we have that \cite[Remark 13.4.2]{buttazzo2006variational}
\begin{equation}\label{5.16}
 \liminf_{j\rightarrow \infty} \mathcal{H}^2(S_{{\bf n}^j})\geqslant \mathcal{H}^2(S_{\bf n}).
\end{equation}
The second result \cite[Thm 5.8]{ambrosio2000functions} gives us the lower semicontinuity of the Oseen-Frank energy so that we have
\begin{equation}\label{5.17}
\liminf_{j\rightarrow \infty}\int_{\Omega}w(\vn^j,\nabla\vn^j)\,dx\geqslant \int_{\Omega}w(\vn,\nabla\vn)\,dx.
\end{equation}
By combining \eqref{5.16} and \eqref{5.17} we see that we have proved the lower semicontinuity 
\begin{equation}\label{5.18}
 \liminf_{j\rightarrow\infty} \mathscr{F}(\vn^j)\geqslant \mathscr{F}(\vn).
\end{equation}
So to conclude existence we just need that $\vn \in \mathcal{A}$. Firstly we require that $\vn\in\mathbb{S}^2\cup \left\{ 0 \right\}$. This is true because weak$^*$ convergence in $BV$ implies convergence in $L^1$ so there is a subsequence converging almost everywhere. The set $\mathbb{S}^2 \cup \left\{ 0 \right\}$ is closed, hence $\vn(x) \in \mathbb{S}^2 \cup \left\{ 0 \right\}$. We also need to prove that
\begin{equation}\label{5.20}
 \vn={\bf g} \quad \text{on}\quad \Omega_{\epsilon}\backslash\overline{\Omega}.
\end{equation}
The same reasoning applies here. The $L^1$ convergence of $\vn^j$ implies pointwise convergence for a subsequence, this readily implies \eqref{5.20}. Hence $\vn \in \mathcal{A}$ and $I(\vn) = \inf_{\vm \in\mathcal{A}} I(\vm)$.

\end{proof}

\vspace{3 mm}

This establishes the existence of a minimiser of this problem for all values of the cholesteric twist $t$. It is clear that this proof easily applies to the standard liquid crystal problem where we have a cuboid domain and boundary conditions on the top and bottom faces. In other words when
\begin{equation}\label{5.21}
 \Omega_{\epsilon}= (-L_1,L_1) \times (-L_2,L_2) \times (-\epsilon,1+\epsilon)
\end{equation}
and
\begin{equation}\label{5.22}
 \mathcal{A}= \left.\left\{ \,\vn\in SBV^2(\Omega_{\epsilon},\mathbb{S}^2\cup \left\{ 0 \right\} )\,\right|\,\vn={\bf e}_1\,\,\text{if}\,\,z\in(-\epsilon,0),\,\,\vn={\bf e}_3\,\,\text{if}\,\,z\in(1,1+\epsilon)\,\right\}.
\end{equation}

%
%
%
%
%
%

\subsection{Euler-Lagrange Equation}

The logical progression, once we know that minimisers exist, is to find the Euler-Lagrange equation for this problem so that we can begin to find candidate minimisers. This is a relatively simple affair when studying problems in Sobolev spaces; we use an outer variation to derive an integral equation which is the weak form of some set of partial differential equations. However the discontinuity set $S_{\vn}$ creates additional issues. For example if we were to perform an outer variation with some smooth function $\vphi$ the discontinuity set would remain unchanged. To overcome this problem we will be using an inner variation to derive the Euler-Lagrange equation. Then we will show that if the discontinuity set $\overline{S}_{\vn}$ has sufficient regularity we can write the Euler-Lagrange equation in other forms.

\vspace{3mm}

The following progression of results mirrors those found in \cite[Section 7.4]{ambrosio2000functions} but we are investigating a different Lagrangian. Therefore we will utilise some results from their work and adapt their arguments to our situation. The functional we will be using throughout this section is given by
\begin{equation}\label{5.23}
\mathscr{F}(\vn) = \int_\Omega |\nabla \vn|^2+2t\vn\cdot\nabla\times \vn + t^2|\vn|^2\, dx + K\mathcal{H}^{d-1} \left( S_\vn \right),
\end{equation}
with the set of admissible functions as given in \eqref{5.22}.
 
\begin{theorem}\label{4: Euler Lagrange 1}
 Let $\vn \in \mathcal{A}$ be a minimiser of \eqref{5.23}. Then
 \begin{equation}\label{5.24}
 \begin{split}
  0= &\int_\Omega \left[ |\nabla \vn|^2 + 2t\vn\nabla\times\vn +t^2|\vn|^2 \right] \text{div}\, \vphi - 2\nabla\vn : \left( \nabla \vn \nabla \vphi \right) - 2t\nabla \vphi : {\bf A}\,dx \\
  &+K\int_{S_\vn} \text{div}^{S_\vn}\, \vphi\, d\mathcal{H}^{d-1} 
 \end{split}
 \end{equation}
 for any $\vphi \in C^1_0 \left( \Omega, \mathbb{R}^d\right)$, where ${\bf A}_{i,j} = \left( \frac{ \partial \vn}{\partial x_i} \times \vn \right)_j$.
 
\end{theorem}

\begin{proof}
 
As mentioned above we will be using an inner variation to derive \eqref{5.24}. Take a $\vphi  \in C^1_0 \left( \Omega, \mathbb{R}^d\right)$. Then for $\epsilon > 0$ sufficiently small, the map $\vpsi_\epsilon (x) := x + \epsilon \vphi(x)$ is a diffeomorphism of $\Omega$ onto itself. Therefore if we let $y = \vpsi_\epsilon (x)$ and define $\vn_\epsilon (y) := \vn\left( \vpsi_\epsilon^{-1} (y) \right)$, we know that 
\begin{equation}\label{5.25}
 \mathscr{F}(\vn_\epsilon) \geqslant \mathscr{F}(\vn).
\end{equation}
In other words 
\begin{equation}\label{5.26}
 \left.\frac{d}{d\epsilon} \mathscr{F}(\vn_\epsilon) \right|_{\epsilon=0}=0. 
\end{equation}
In order to use \eqref{5.26} we will find an expression for $\mathscr{F}(\vn_\epsilon) - \mathscr{F}(\vn)$. By performing a change of variables we find
\begin{equation}\label{5.27}
\begin{split}
 &\int_\Omega |\nabla \vn_\epsilon (y)|^2+2t\vn_\epsilon(y)  \cdot \nabla \times \vn_\epsilon (y)+t^2|\vn_\epsilon(y)|^2\,dy-\int_\Omega |\nabla \vn|^2+2t\vn\cdot \nabla \times \vn+t^2|\vn|^2\, dx \\
 =& \int_\Omega \left[ \left| \nabla \vn (x) \nabla \vpsi_\epsilon^{-1} \left( \psi_\epsilon(x) \right) \right|^2 + 2 t \vn_i(x)\epsilon_{ijk} \left( \nabla \vn \nabla \vpsi_\epsilon^{-1} \left( \vpsi_\epsilon (x) \right)\right)_{kj} + t^2 |\vn(x)|^2 \right] \left| \text{det} \nabla \vpsi_\epsilon (x) \right|\,dx \\
 &-\int_\Omega|\nabla\vn|^2-2t\vn\cdot\nabla \times \vn +t^2|\vn|^2\,dx.
\end{split}
\end{equation}
But we only need to consider terms up to order $\epsilon$ and we can easily calculate that 
\begin{equation}\label{5.28}
\begin{split}
 \nabla \vpsi_\epsilon^{-1}\left( \vpsi_\epsilon(x) \right) = \left( I+\epsilon\nabla \vphi(x)\right)^{-1} = I-\epsilon\nabla \vphi(x)+O(\epsilon^2)\\
 \text{det} \nabla \vpsi_\epsilon (x) = \text{det}\left( I + \epsilon \nabla \vphi(x) \right) = 1+\epsilon \text{div}\, \vphi(x)+O(\epsilon^2).
\end{split}
\end{equation}
Applying these approximations to \eqref{5.27} allows us to deduce that
\begin{equation}\label{5.29}
 \begin{split}
  &\mathscr{F}(\vn_\epsilon)-\mathscr{F}(\vn) \\
  =& \int_\Omega \left[ |\nabla\vn(I-\epsilon \nabla \vphi)|^2+2t\vn_i \epsilon_{ijk} \left( \nabla \vn (1-\epsilon \nabla \vphi) \right) +t^2 |\vn|^2 \right] (1+\epsilon \text{div}\, \vphi)\,dx \\
  &- \int_\Omega|\nabla\vn|^2-2t\vn\cdot\nabla \times \vn+t^2|\vn|^2\,dx+ O(\epsilon^2) + K\left(\mathcal{H}^{d-1} (S_{\vn_\epsilon})-\mathcal{H}^{d-1} (S_{\vn}) \right) \\
  =& \epsilon \int_\Omega \left[ |\nabla \vn|^2+2t\vn\cdot\nabla \times \vn+t^2|\vn|^2\right] \text{div}\, \vphi - 2\nabla \vn:\left(\nabla \vn \nabla \vphi\right) -2t \nabla \vphi:{\bf A}\, dx+O(\epsilon^2) \\
  &+K\left(\mathcal{H}^{d-1} (S_{\vn_\epsilon})-\mathcal{H}^{d-1} (S_{\vn}) \right),
 \end{split}
\end{equation}
where ${\bf A}$ is given in the statement. To finish the proof we just need to apply \cite[Thm 7.31]{ambrosio2000functions} concerning the first variation of the area to deduce that 
\begin{equation}\label{5.30}
 K\left(\mathcal{H}^{d-1} (S_{\vn_\epsilon})-\mathcal{H}^{d-1} (S_{\vn}) \right) = \epsilon K \int_{S_\vn} \text{div}^{S_\vn}\, \vphi\, d\mathcal{H}^{d-1}+ O(\epsilon^2).
\end{equation}
Then dividing \eqref{5.29} by $\epsilon$ and taking the limit as $\epsilon\rightarrow 0$ gives us the statement.

\end{proof}

We cannot easily formulate \eqref{5.24} as the weak form of some strong equation due to the unknown regularity of the discontinuity set $S_\vn$. Therefore in the following results we will assume some a-priori regularity of the discontinuity set, or the function itself, and find equivalent forms of the Euler-Lagrange equation given the regularity condition.

\begin{proposition}\label{4: Euler Lagrange 2}
 Suppose $\vn \in \mathcal{A}$ is a minimiser of \eqref{5.23}. Let $A\subset\subset \Omega$ and suppose that 
 \begin{equation}\label{5.31}
  \overline{S}_\vn\cap A = \left\{\, \left. (z,\vphi(z))\,\right|\, z\in D\,\right\},
 \end{equation}
 where $D\subset \mathbb{R}^{d-1}$ is an open set and $\vphi\in C^{1} \left( D,\mathbb{R} \right)$. Then 
 \begin{equation}\label{5.32}
  \int_{A\setminus \overline{S}_\vn} \nabla \vv :\nabla \vn +t\left(\vv \cdot\nabla \times \vn+\vn\cdot\nabla \times \vv \right) - (\vn\cdot \vv)\left( |\nabla \vn|^2+2t\vn\cdot \nabla \times \vn\right) \,dx =0,
 \end{equation}
 for all $\vv \in C^\infty_0 \left( A , \mathbb{R}^k \right)$.
\end{proposition}

\begin{proof}
 
 The local regularity of the discontinuity gives us the decomposition that 
 \begin{equation}\label{5.33}
  A = A^+ \cup A^- \cup \left( A \cap \overline{S}_\vn \right)
 \end{equation}
 where $A^+ = \left\{ \, (z,t)\in A\,|\, t>\vphi(z)\, \right\}$ and $A^- = \left\{ \, (z,t)\in A\,|\, t<\vphi(z)\, \right\}$ are open sets. We take some $\vv \in C^\infty_0 \left( A , \mathbb{R}^k \right)$ and we will consider the outer variation 
 \begin{equation}\label{5.34}
  \vn_\epsilon(x) :=\left\{ 
  \begin{array}{lcc}
   P(\vn+\epsilon \vv) & & x\in A^+ \\ \vn(x) && x\in A^- 
  \end{array} \right. ,
 \end{equation}
 where 
 \begin{equation}\label{5.34.1}
   P(x) = \left\{ \begin{array}{lcc} \frac{x}{|x|} & & |x| \in \left( \frac{2}{3},\frac{4}{3} \right) \\ 0 & & |x|< \frac{1}{3} \end{array} \right.
  \end{equation}
  is the projection map onto $\mathbb{S}^2 \cup \left\{ 0 \right\}$. Using this variation we can certainly say that $\left.\frac{d}{d\epsilon} \mathscr{F}(\vn_\epsilon) \right|_{\epsilon=0} =0$. So by arguing in the same manner as variational problems in Sobolev spaces we find 
 \begin{equation}\label{5.35}
 \begin{split}
 0&= \int_{\left\{ \vn=0\right\}} 0\,dx +\int_{\left\{ \vn \in \mathbb{S}^2 \right\}} \nabla \vv :\nabla \vn +t\left(\vv \cdot\nabla \times \vn+\vn\cdot\nabla \times \vv \right) - (\vn\cdot \vv)\left( |\nabla \vn|^2+2t\vn\cdot \nabla \times \vn\right) \,dx \\
   &=\int_{A^+} \nabla \vv :\nabla \vn +t\left(\vv \cdot\nabla \times \vn+\vn\cdot\nabla \times \vv \right) - (\vn\cdot \vv)\left( |\nabla \vn|^2+2t\vn\cdot \nabla \times \vn\right) \,dx.
   \end{split}
 \end{equation}
 We obtain a similar equation in $A^-$ using the same logic. Integrating by parts on the terms involving the gradients of $\vv$ we find that this is the weak form of
 \begin{equation}\label{5.36}
  \begin{array}{ccl}
   \Delta \vn -2t\nabla \times \vn = -\vn \left( |\nabla \vn|^2+2t\vn\cdot\nabla \times \vn \right) &\quad&  \text{in}\,\, A^+ \\ &&\\
   \frac{\partial \vn}{\partial \nu}+t\vn \times \bm\nu = 0 &\quad& \text{on} \,\, \partial A^+\cap \overline{S}_\vn,
  \end{array}
 \end{equation}
where $\bm\nu$ is the normal to the boundary.
 
\end{proof}

This is a reassuring proposition since it proves that away from the jump set, the minimiser $\vn$ must satisfy the same equation as would be satisfied by a Sobolev director field on the entirety of the domain. However by assuming a little more regularity for the function and its jump set we can strengthen the Euler-Lagrange equation still further.

\begin{theorem}\label{4: Euler Lagrange 3}
 
 Suppose $\vn \in \mathcal{A}$ is a minimiser of \eqref{5.23}. Let $A\subset \subset \Omega$ and suppose that $\vn \in W^{2,2}\left( A\setminus \overline{S_\vn}\right)$ and $\overline{S}_\vn \cap A$ is given as in \eqref{5.31}, where $D\subset \mathbb{R}^{d-1}$ is an open set and $\vphi \in C^{1,\gamma} \left( D, \mathbb{R} \right)$. Then for every $\eta \in C^\infty_0(A,\mathbb{R}^d)$ 
 \begin{equation}\label{5.37}
  \int_{{S}_\vn \cap A} \left[ |\nabla \vn|^2 + 2t \vn \cdot \nabla \times \vn + t^2 |\vn|^2 \right]^{\pm} \nu \cdot \eta \, d\mathcal{H}^{d-1} = K\int_{{S}_\vn\cap A} \text{div}^{S_\vn} \eta\, d\mathcal{H}^{d-1},
 \end{equation}
 where $\nu$ is the upper normal to $\overline{S}_\vn \cap A$.
 
\end{theorem}

\begin{proof}
 
 The essence of the proof is to use integration by parts on equation \eqref{5.24}. We recall that the matrix ${\bf A}$ is defined by ${\bf A}_{i,j} = \sum_{a,b} \epsilon_{jab}\vn_{a,i}\vn_b$. In order to do so we note initially that $\vn \in W^{2,2}(A^+)\cap W^{2,2}(A^-)\cap \mathcal{A}$. We begin with the following calculation 
 \begin{equation}\label{5.38}
  \begin{split}
   &\int_{A^+} \left[ |\nabla \vn|^2+2t\vn\cdot\nabla\times \vn + t^2 |\vn|^2 \right] \text{div} \eta - 2\nabla\vn:\left( \nabla\vn \nabla \eta\right) -2t \nabla \eta: {\bf A}\,dx \\
   =& \int_{S_\vn \cap A} -\left[ |\nabla \vn|^2 + 2t\vn\cdot\nabla\times\vn + t^2 |\vn|^2  \right]^+ \nu\cdot \eta + 2\vn_{i,j}\vn_{i,a}\eta_{a}\nu_j + 2t\eta_a {\bf A}_{a,j}\nu_j \,  d\mathcal{H}^{d-1} \\
   &+\int_{A^+} -\eta_a \frac{\partial}{\partial x_a}\left(\vn_{i,j}\vn_{i,j} + 2t\vn_{i}\epsilon_{ijk}\vn_{k,j} +t^2 \vn_i\vn_i \right) + \eta_a \frac{\partial}{\partial x_j} \left( 2\vn_{i,j}\vn_{i,a} +2t {\bf A}_{a,j} \right)\, dx.
  \end{split}
 \end{equation}
In order to simplify this we note that
\begin{equation}\label{5.39}
 \vn_{i,j}\vn_{i,a}\eta_a \nu_j +t\eta_a {\bf A}_{a,j}\nu_j = \vn_{i,a}\eta_a \left( \frac{\partial \vn_i}{\partial \nu} + t (\vn \times\nu)_i \right) = 0.
\end{equation}
This holds because of the natural boundary condition from equation \eqref{5.36}. Furthermore, using the Euler-Lagrange equation from \eqref{5.36} we can also deduce 
\begin{equation}\label{5.40}
\begin{split}
 &\frac{\partial}{\partial x_a}\left(\vn_{i,j}\vn_{i,j} + 2t\vn_{i}\epsilon_{ijk}\vn_{k,j} \right) +\frac{\partial}{\partial x_j} \left( 2\vn_{i,j}\vn_{i,a} +2t {\bf A}_{a,j} \right) \\
 =& 2(\vn_{i,jj}\vn_{i,a}-2t\vn_{i,a}\epsilon_{ijk}\vn_{k,j})\\
 =& 2\left[ |\nabla \vn|^2 +2t \vn \cdot\nabla \times \vn\right] \sum_i\vn_{i,a}\vn_i =0 .
\end{split}
\end{equation}
Therefore
\begin{equation}\label{5.41}
\begin{split}
 &\int_{A^+} \left[ |\nabla \vn|^2+2t\vn\cdot\nabla\times \vn+ t^2 |\vn|^2 \right] \text{div} \eta - 2\nabla\vn:\left( \nabla\vn \nabla \eta\right) -2t \nabla \eta: {\bf A}\,dx \\
 =& \int_{S_\vn \cap A} -\left[ |\nabla \vn|^2 + 2t\vn\cdot\nabla\times\vn+ t^2 |\vn|^2  \right]^+ \nu\cdot \eta dx,
\end{split}
\end{equation}
and a similar equation holds in the region $A^-$. This gives us the assertion.

\end{proof}

\begin{remark}
 In \cite{ambrosio2000functions} they were able to extend Proposition \ref{4: Euler Lagrange 3} without assuming the extra regularity that $\vn \in W^{2,2}(A\setminus \overline{S_\vn})$. In order to be able to do the same here we would require some a-priori knowledge about the regularity of weak solutions to \eqref{5.36} which we do not have.
\end{remark}

Our final remark on the Euler-Lagrange equation comes courtesy of the divergence theorem on manifolds.

\begin{remark}\cite[Remark 7.39]{ambrosio2000functions}
 If in addition to the assumptions of Theorem \ref{4: Euler Lagrange 3}, the set $A\cap \overline{S}_\vn$ can be written as the graph of a $C^2$ function $\vphi$, then we find that \eqref{5.37} can be written as 
 \begin{equation}\label{5.42}
  -K\text{div}\left( \frac{\nabla \vphi}{\sqrt{1+|\nabla\vphi|^2}} \right) = KH = \left[ |\nabla \vn|^2 + 2t\vn \cdot \nabla \times \vn + t^2 |\vn|^2 \right]^{\pm}
 \end{equation}
 where $H$ is the mean scalar curvature of the set $\overline{S}_\vn \cap A$.

\end{remark}

\subsection*{Local Minimisers}

One final aspect of bounded variation problems to introduce in this section which we have not covered up to this point, is the notion of local minimisers. It is more difficult to define local minimisers for bounded variation problems because the discontinuity set needs to varied simultaneously with the elastic energy. We follow the definition from \cite[Chapter 7]{ambrosio2000functions}.

\begin{definition}[Local minimiser]\label{4: local minimiser}
$\vn \in \mathcal{A}$ is a local minimiser of $I$ if 
\begin{equation}\label{5.42.1}
 \int_A w(\vn,\nabla \vn) \,dx +K\mathcal{H}^2\left( A \cap S_\vn \right) \leqslant \int_A w(\vm,\nabla \vm)\,dx + K\mathcal{H}^2\left( A \cap S_\vm \right)
\end{equation}
whenever $\vm \in \mathcal{A}$ and $\left\{ \vn \neq \vm \right\} \subset\subset A \subset \subset \Omega$.

\end{definition}

This definition is local in the sense that, loosely speaking, the candidate needs only to be a minimiser amongst admissible functions with the same boundary values. We cannot use the notions of strong and weak local minimisers as we could for Sobolev spaces because it would result in absurd predictions. For example any function $\vn \in SBV(\Omega,\mathbb{S}^2)$ with $\nabla \vn = 0$ and any number of jumps of fixed magnitude would be a strong local minimiser of the nematic problem.

\begin{remark}\label{4: remark local min equiv global min}
 It should be noted that if we set up a bounded variation problem where we wish to apply boundary conditions around the entirety of the boundary then $\overline{\Omega} \subset \Omega_\epsilon$. In this situation there is no difference between a local minimiser and a global minimiser.
\end{remark}

%
%
%
%
%
%

\section{Lavrentiev Phenomena for Nematic Problems}\label{4: sec lavrentiev}

The natural step to take once we have the Euler-Lagrange equations established is to find solutions with a non-trivial jump set. Clearly any $\vn \in W^{1,2}\left(\Omega, \mathbb{S}^2 \right)$ which satisfies the classical Euler-Lagrange equation also satisfies those in the previous section. In this section we will examine two problems associated with the nematic energy
\begin{equation}\label{5.43}
 I(\vn) = \int_\Omega |\nabla \vn |^2 \,dx+K\mathcal{H}^{d-1}(S_\vn).
\end{equation}
For the first we will look for minimisers of \eqref{5.43} when the domain is a unit ball in two or more dimensions with radial boundary conditions. We will show that by removing a small core of isotropic liquid crystal we can lower the energy and prove a Lavrentiev gap phenomenon between Sobolev and SBV functions. In the second we set the problem in a cuboid with constant boundary conditions on the top and bottom faces which are perpendicular to each other. In this situation we can show another Lavrentiev gap phenomenon by finding a global minimiser in both the Sobolev and SBV spaces.

\begin{theorem}\label{4: Lavrentiev gap 1}
 Consider the energy functional $I$ as given in \eqref{5.43}. Let $d\geqslant 2$, if $\Omega = \left.\left\{\, x\in \mathbb{R}^d\, \right|\, |x|<1\, \right\}$ and the sets of admissible functions are given by
 \begin{equation}\label{5.44}
  \mathcal{A}:=  \left\{ \left.\, \vn \in SBV^2\left(\Omega,\mathbb{S}^{d-1}\cup \left\{0\right\} \right) \, \right|\, \vn(x)=x\,\, \text{on}\,\, \partial \Omega\,\right\},
 \end{equation}
 and
 \begin{equation}\label{5.45}
  \mathcal{B}:=\left\{\,\left.  \vn \in W^{1,2}\left(\Omega,\mathbb{S}^{d-1}\right) \, \right|\, \vn(x)=x\,\, \text{on}\,\, \partial \Omega\,\right\}.
 \end{equation}
 Then if $K>1$,
 \begin{equation}\label{5.46}
  \inf_{\vn \in \mathcal{A}} I(\vn) < \inf_{\vn \in \mathcal{B}} I(\vn).
 \end{equation}

\end{theorem}

\begin{proof}

 The infimum over the admissible set $\mathcal{B}$ has been extensively studied and Lin \cite{lin1987remark}, among others, showed that when $d\geqslant 3$, $\vn(x) = \frac{x}{|x|}$ is a global minimiser of $I$ with an energy of $\frac{d-1}{d-2}S_{d-1}$, where $S_{d-1}$ is the surface area of the sphere $\mathbb{S}^{d-1}$. In the case $d=2$, $\mathcal{B}$ is empty so we consider the infimum from \eqref{5.46} to be infinite. To prove the statement we merely need to find a function $\vn \in \mathcal{A}$ such that $I(\vn) < \frac{d-1}{d-2}S_{d-1}$. However we will go further and find a solution of the Euler-Lagrange equation with this property. Let $\alpha \in (0,1)$ and consider the function 
 \begin{equation}\label{5.47}
  \vn_\alpha (x) = \left\{ 
  \begin{array}{ccl} 
   \frac{x}{|x|} & \quad & |x|>\alpha \\ 
   && \\
   0 & \quad & |x|<\alpha
  \end{array}
 \right. .
 \end{equation}
 It is clear that $\vn_\alpha \in \mathcal{A}$ and $S_{\vn_\alpha} = \left\{\, |x| = \alpha \, \right\}$. This set (up to a rotation) can easily be locally written as the graph of the $C^2$ function 
 \begin{equation}\label{5.48}
  \phi(x) =  \sqrt{\alpha^2-\sum_{i=0}^{d-1} x_i^2}.
 \end{equation}
Then when we substitute this into \eqref{5.42} we find
\begin{equation}\label{5.49}
\begin{split}
 -K \text{div}\left( \frac{\nabla \vphi}{\sqrt{1+|\nabla\vphi|^2}} \right)  
 =& -K\text{div} \left( -\frac{x}{\alpha} \right) \\
 =& \frac{(d-1)K}{\alpha} \\
 =& K\alpha \left[ |\nabla \vn_\alpha |^2 \right]^{\pm}.
 \end{split}
\end{equation}
Therefore the Euler-Lagrange equation is satisfied when $\alpha = \frac{1}{K}$ in every dimension $d\geqslant 2$. The associated energy of this state is given by
\begin{equation}\label{5.50}
\begin{split}
 I\left( \vn_{\frac{1}{K}} \right)  &= \int_{\frac{1}{K} < |x| < 1} \frac{(d-1)}{|x|^2}\, dV +K\left[ K^{-(d-1)} S_{d-1} \right] \\
 &=\left\{
 \begin{array}{ccl}
  \frac{d-1}{d-2} S_{d-1} - \frac{S_{d-1}}{(d-2)K^{d-2}} & \quad & d\geqslant 3 \\ & & \\
  \log K + \frac{2\pi}{K} & \quad & d=2
 \end{array} \right.\\
 &=\left\{
 \begin{array}{ccl}
  \inf_{\vm \in \mathcal{B}} I(\vm) - \frac{S_{d-1}}{(d-2)K^{d-2}} & \quad & d\geqslant 3 \\ & & \\
  \log K + \frac{2\pi}{K} & \quad & d=2.
 \end{array} \right.
\end{split}
\end{equation}
This proves the Lavrentiev gap phenomenon and we have also found a strong candidate for the minimiser of the bounded variation problem: $\vn_{\frac{1}{K}}$.
 
\end{proof}

For our second example we consider a standard nematic liquid crystal problem set in a cuboid domain in our new SBV setting. We look to minimise
\begin{equation}\label{5.51}
 I(\vn) = \int_\Omega |\nabla \vn|^2\,dx +K\mathcal{H}^{d-1}(S_\vn)
\end{equation}
where $\Omega = (-L_1,L_1)\times(-L_2,L_2)\times(0,d)$ and the set of admissible functions is given by
\begin{equation}\label{5.51.1}
 \mathcal{A} := \left\{\,\left. \vn \in SBV^2 \left( \Omega, \mathbb{S}^2\cup \left\{ 0 \right\}\right) \, \right|\, \vn|_{z=0} = {\bf e}_1,\,\, \vn|_{z=d} = {\bf e}_3\, \right\}.
\end{equation}
For the next two results we need to use the result about the one-dimensional restriction of SBV functions, Theorem \ref{4: 1D SBV sections}, given in the preliminaries.

\begin{theorem} \label{4: Lavrentiev gap: Nematics}
 If $K>\frac{\pi^2}{4d}$ then the unique global minimiser of \eqref{5.51} over \eqref{5.51.1} is given by
 \begin{equation}\label{5.53}
  \vn^*=\left(\begin{array}{c} \cos\left(\frac{\pi z}{2d}\right) \\0\\ \sin\left(\frac{\pi z}{2d}\right)\end{array}\right).
 \end{equation}
 If $K<\frac{\pi^2}{4d}$ then the global minimisers are given by the one parameter family 
 \begin{equation}\label{5.54}
  \vn_\alpha'=\left\{\begin{array}{lcc} {\bf e}_1&&z\in(0,\alpha)\\ && \\{\bf e}_3&&z\in(\alpha,1),\end{array}\right. 
 \end{equation}
 for $\alpha\in(0,1)$. In the critical case both $\vn_\alpha'$ and $\vn^*$ are global minimisers.
\end{theorem}

\begin{proof}

We take an arbitrary $\vn\in\mathcal{A}$. Then we define the set
\begin{equation}\label{5.55}
 \Gamma({S_{\vn}})=\left\{\, (x,y)\,|\,(x,y,z)\in {S_{\vn}}\,\, \text{for some}\,\, z \,\right\}.
\end{equation}
This is the projection of ${S_{\vn}}$ onto the $(x,y)$-plane. Then we have the following inequality
\begin{equation}\label{5.56}
  I(\vn)\geqslant \int_{\Gamma({S_{\vn}})^c}\int_0^d |\nabla\vn|^2\,dzdxdy+K\mathcal{H}^2(\Gamma({S_{\vn}})),
\end{equation}
where $\Gamma({S_{\vn}})^c = (-L_1,L_1)\times(-L_2, L_2) \setminus \Gamma({S_{\vn}})$. From \eqref{5.56} we can use our knowledge from previous work to deduce a further inequality. By Theorem \ref{4: 1D SBV sections} it is clear that for $\mathcal{H}^2$ almost every $(x,y) \in \Gamma \left( S_\vn \right)$, that $\vm(\cdot):=\vn(x,y,\cdot) \in SBV^2\left( (-\epsilon,d+\epsilon),\mathbb{S}^2 \cup \left\{ 0 \right\} \right)$. However we know that $S_\vm = \emptyset$. Therefore
\begin{equation}\label{5.56.1}
 \vn(x,y,\cdot) \in W^{1,2}\left((0,d),\mathbb{S}^2\right),\,\, \vn(x,y,0)={\bf e}_1,\,\,\vn(x,y,d) = {\bf e}_3.
\end{equation}
By using a Poincar\'{e} inequality with a specific constant (see \cite{dym1972fourier}) to know that any function satisfying the conditions in \eqref{5.56.1} also satisfies
\begin{equation}\label{5.57}
 \int_0^d |\nabla\vn(x,y,z)|^2\,dz\geqslant \frac{\pi^2}{4d}.
\end{equation}
By applying this reasoning to equation \eqref{5.56} we get 
\begin{equation}\label{5.58}
 I(\vn)\geqslant \frac{\pi^2}{4d}\mathcal{H}^2(\Gamma(S_{\vn})^c)+K\mathcal{H}^2(\Gamma(S_{\vn})).
\end{equation}
Now we split the proof into the three cases. First we suppose that $K>\frac{\pi^2}{4d}$. Then have the following inequality
\begin{equation}\label{5.59}
 I(\vn)\geqslant\frac{\pi^2}{4d}(4L_1L_2),
\end{equation}
with a necessary condition for equality being 
\begin{equation}\label{5.60}
 \mathcal{H}^2(\Gamma(S_{\vn}))=0. 
\end{equation}
Suppose that $\vn$ satisfies \eqref{5.60} and achieves equality in \eqref{5.59}. These two assumptions allow us to immediately deduce
\begin{equation}\label{5.61}
 \int_0^1 |\nabla\vn|^2\,dx \geqslant \frac{\pi^2}{4d},
\end{equation}
for $\mathcal{H}^2$ almost every, $(x,y) \in (-L_1,L_1)\times(-L_2,L_2)$. Hence 
\begin{equation}\label{5.62}
 I(\vn)\geqslant \frac{\pi^2L_1L_2}{d}+K\mathcal{H}^2(S_{\vn})\geqslant \frac{\pi^2L_1L_2}{d},
\end{equation}
with equality only if $\mathcal{H}^2(S_\vn)=0$. Therefore equality can only be achieved if and only if $\vn = \vn^*$, the unique minimum over Sobolev functions. The second case where $K<\frac{\pi^2}{4d}$ is a very similar one. We have the inequality (from \eqref{5.58}) that 
\begin{equation}\label{5.63}
 I(\vn) \geqslant K(4L_1L_2),
\end{equation}
with a necessary condition for equality being
\begin{equation}\label{5.64}
  \mathcal{H}^2(\Gamma(S_{\vn})^c)=0. 
\end{equation}
As before, if we suppose that $\vn$ satisfies \eqref{5.64}, then we deduce that equality in \eqref{5.63} implies
\begin{equation}\label{5.64.1}
 I(\vn) = \int_\Omega |\nabla \vn|^2\,dx + K\mathcal{H}^2 \left( S_\vn \right) \geqslant \int_\Omega |\nabla \vn|^2 \,dx + K\mathcal{H}^2 \left( \Gamma(S_\vn) \right) \geqslant 4KL_1L_2
\end{equation}
which has equality when $\nabla \vn =0$ and $\mathcal{H}^2 (S_\vn) = \mathcal{H}^2 (\Gamma (S_\vn))$. These relations only hold when
\begin{equation}\label{5.65}
 S_{\vn}=\left\{ \,(x,y,z)\,|\,z=\alpha\,\right\}.
\end{equation}
Hence $\vn = \vn'_\alpha$ for some $\alpha \in (0,1)$. The critical case $K=\frac{\pi^2}{4d}$, immediately follows from \eqref{5.58}.

\end{proof}

%
%
%
%
%
%

\section{Cholesteric Problem}\label{section: cholesteric problem}

We are now in a position to turn to our attentions towards cholesteric liquid crystals. To show the usefulness of our new bounded variation approach it is important to know how it predictions differ from those made using the classical Sobolev space. To this end we present some results from \cite{Bedford2014} and then study the same problem within this new framework. In the Sobolev framework of \cite{Bedford2014}, the aim was to minimise
\begin{equation}\label{5.66}
 I(\vn) = \int_\Omega |\nabla \vn|^2+2t\vn \cdot\nabla \times \vn + t^2 |\vn|^2 \, dx + K\mathcal{H}^{2}(S_\vn),
\end{equation}
over the following set of admissible functions
\begin{equation}\label{5.76.1}
 \mathcal{A} := \left\{ \,\left. \vn \in W^{1,2} \left( \Omega, \mathbb{S}^2 \right) \, \right|\, \vn|_{z=0} = {\bf e}_1,\,\, \vn|_{z=d} = {\bf e}_3,\,\, \vn|_{x=-L_1} = \vn|_{x=L_1},\,\, \vn|_{y=-L_2} = \vn|_{y=L_2}\, \right\}.
\end{equation}
Initially functions of just the height $z$ were considered. In terms of its Euler angles, any function $\vn = \vn(z) \in \mathcal{A}$ can be written as 
\begin{equation}\label{5.76.2}
 \vn(z) = \left( \begin{array}{c} \cos(\theta(z)) \cos(\phi(z)) \\ \cos(\theta(z)) \sin(\phi(z)) \\ \sin(\theta(z)) \end{array} \right).
\end{equation}
In this situation the following two results were proved

\begin{theorem}\label{thm: example 1}
 The function $\theta^*(z)$ given implicitly by the equation
 \begin{equation}\label{5.76.3}
  \int_0^{\theta^*} \frac{1 }{\left( D - t^2 \cos^2 u \right)^\frac{1}{2} } \, du = 1 \quad \theta^*(1) = \frac{\pi}{2},
 \end{equation}
 with $\phi^*(z) = tz$, determines a director field $\vn^*(z)$ which is the unique global minimum of \eqref{5.66} over admissible functions \eqref{5.76.1} which depend only on $z$.

\end{theorem}

\begin{theorem}\label{thm: example 2}
 There exists some $t^*>0$ such that one-dimensional minimum $\vn^*(z)$ from the previous theorem is the global minimum of \eqref{5.66} over \eqref{5.76.1} if $t<t^*$.
\end{theorem}

However to the best of our knowledge it is a completely open analytical problem to find minimisers in the high chirality limit. We note that because we have scaled our domain to a height $1$, the variable $t$ is proportional to the commonly cited confinement ratio, $d/P$, where $d$ is the height of the domain and $P$ is the pitch of the cholesteric. We will now show that in our bounded variation framework we can extend the analysis of Theorems \ref{thm: example 1} and \ref{thm: example 2}. Hence, for the rest of the paper, we now assume that we are looking to minimise \eqref{5.66} over the set of admissible functions
\begin{equation}\label{5.76.4}
 \mathcal{A} := \left\{ \,\left. \vn \in SBV^{2} \left( \Omega, \mathbb{S}^2\cup \left\{ 0 \right\} \right) \, \right|\, \vn|_{z=0} = {\bf e}_1,\,\, \vn|_{z=d} = {\bf e}_3,\,\, \vn|_{x=-L_1} = \vn|_{x=L_1},\,\, \vn|_{y=-L_2} = \vn|_{y=L_2}\, \right\}.
\end{equation}
In the new setting we can demonstrate that if the cholesteric twist is small enough, the minimum of the variational problem is the smooth one variable minimiser. Whereas if the cholesteric twist is large enough, the minimiser of the problem cannot be a function of a single variable.

\begin{theorem}\label{4: Cholesteric small t large K global min}
 There exists constants $t^*,K^*>0$ such that if $t<t^*$ and $K>K^*$, then the unique global minimum of \eqref{5.66} over \eqref{5.76.4} is the smooth function $\vn(z)$ defined in Theorem \ref{thm: example 1}

\end{theorem}

\begin{proof}
 
 During this proof, $C$ will represent any arbitrary positive constant. As such its value may not be constant during calculations. Let $\vn$ be the smooth minimiser defined in Theorem \ref{thm: example 1} and let $\vm \in \mathcal{A}$ be arbitrary. The main work of the proof is to show that 
 \begin{equation}\label{5.1.2}
 \begin{split}
  I(\vm)-I(\vn) \geqslant & \int_\Omega |\nabla \vm - \nabla \vn|^2 + 2t \left( \vn -\vm \right) \cdot \nabla \times \left( \vn - \vm \right) -\lambda |\vn-\vm|^2 +(\lambda+t^2)\left[|\vm|^2 - 1\right]\, dx \\ & - C(1+t) \mathcal{H}^2 \left( S_\vm \right) + K \mathcal{H}^2 \left( S_\vm \right),
 \end{split}
 \end{equation}
 where $\lambda := |\nabla \vn|^2 + 2t\vn \cdot \nabla \times \vn$. We start with a simple rearrangement
 \begin{equation}\label{5.1.3}
  \begin{split}
   I(\vm) - I(\vn) & = \int_\Omega |\nabla \vm|^2 - |\nabla \vn|^2 +2t \vm \cdot \nabla \times \vm -2t \vn \cdot \nabla \times \vn +t^2 |\vm|^2 -t^2 \,dx +K\mathcal{H}^2(S_\vm) \\
   & = \int_\Omega |\nabla \vm - \nabla \vn|^2 +2t(\vm - \vn) \cdot \nabla \times (\vm-\vn) -\lambda |\vn-\vm|^2 + (\lambda+t^2) \left[ |\vm|^2 - 1 \right] \,dx 
   \\ & \quad + 2\int_\Omega \nabla \vn:\nabla \vm + t(\vn\cdot \nabla \times \vm + \vm \cdot \nabla \times \vn ) -\lambda \vm\cdot \vn \, dx+ K\mathcal{H}^2 \left( S_\vm \right). 
  \end{split}
 \end{equation}
 Clearly to show \eqref{5.1.2}, our aim is to prove that the final integral term in the above equation can be estimated by $C(1+t)\mathcal{H}^2 \left( S_\vm \right)$. This is possible because it is almost the weak form of the Euler-Lagrange equation which is satisfied by $\vn$. We know that 
 \begin{equation}\label{5.1.4}
  \Delta \vn -2t\nabla \times \vn + \lambda \vn=0.
 \end{equation}
 Hence by multiplying by $\vm$ and integrating we get
 \begin{equation}\label{5.1.5}
  \int_\Omega \vm \cdot \left( \Delta \vn -2t \nabla \times \vn +\lambda \vn \right)\,dx=0.
 \end{equation}
Integration by parts is a little different in SBV because we have the discontinuity set to deal with. However a generalised Gauss-Green formula still holds.

\begin{theorem}\cite[Thm 10.2.1]{buttazzo2006variational}\label{4: Gauss Green}
If $U\subset \mathbb{R}^d$ is an open set (with Lipschitz boundary), $u \in SBV(\Omega,\mathbb{R})$, and $\vphi \in C^1 \left( \overline{U} , \mathbb{R}^3 \right)$ then
\begin{equation}\label{5.1.6}
 \int_\Omega \vphi \cdot \nabla u \,dx + \int_{S_u} (u^+-u^-)\vphi \cdot {\bm \nu} d\mathcal{H}^{d-1} = \int_\Omega \vphi \cdot Du = -\int_\Omega u \, div\, \vphi\,dx + \int_{\partial \Omega} u \vphi\cdot {\bm \nu} d\mathcal{H}^{d-1}.
\end{equation}

\end{theorem}

By applying the above equation to each term of $\vm$ separately we get 
\begin{equation}\label{5.1.7}
 \int_\Omega \nabla \vm \cdot \nabla \vn \, dx = - \int_\Omega \vm \cdot \Delta \vn\,dx + \int_{\partial \Omega} \vm \cdot (\nabla \vn \cdot {\bm \nu} )\,d\mathcal{H}^2 -\sum_i \int_{S_{\vm_i}} (\vm_i^+-\vm_i^-)\nabla \vn \cdot {\bm \nu}\, d\mathcal{H}^2.
\end{equation}
The term around the boundary can be shown to be zero using our knowledge of $\vn$ and the boundary conditions for $\vm$. On the lateral faces we simply have that $\nabla \vn \cdot {\bm \nu} = 0$ and on the top and bottom faces we have that 
\begin{equation}\label{5.1.8}
  \vm \cdot (\nabla \vn \cdot {\bm \nu} )|_{z=0} = -\vn_{1}'(0) =\left. \theta'\sin(\theta)\cos(tz)-t\cos(\theta)\sin(tz)\right|_{z=0} = 0 
\end{equation}
and
\begin{equation}\label{5.1.9}
  \vm \cdot (\nabla \vn \cdot {\bm \nu} )|_{z=d} = \vn_{3}'(d) = \left.\theta'\cos(\theta)\right|_{z=d} = 0.
\end{equation}
Furthermore the integral over the jump set in \eqref{5.1.7} can easily be estimated by 
\begin{equation}\label{5.1.10}
 \left|\sum_i \int_{S_{\vm_i}} (\vm_i^+-\vm_i^-)\nabla \vn \cdot {\bm \nu}\, d\mathcal{H}^2\right| \leqslant C ||\vn||_{1,\infty} \mathcal{H}^2 (S_\vm) = C \mathcal{H}^2 (S_\vm).
\end{equation}
Hence
\begin{equation}\label{5.1.11}
 \left|\int_\Omega \vm\cdot \Delta \vn +\nabla \vm : \nabla \vn\,dx \right| \leqslant C \mathcal{H}^2(S_\vm).
\end{equation}
Using very similar logic we can also show that 
\begin{equation}\label{5.1.12}
 \left|\int_\Omega \vn\cdot \nabla \times \vm -\vm \cdot \nabla \times \vn\,dx \right| \leqslant C \mathcal{H}^2 (S_\vm).
\end{equation}
By substituting \eqref{5.1.12} and \eqref{5.1.11} into \eqref{5.1.3} we deduce that
\begin{equation}\label{5.1.13}
 I(\vm)-I(\vn) \geqslant \int_\Omega |\nabla \vv|^2 +2t\vv\cdot \nabla \times \vv -\lambda |\vv|^2 +(\lambda+t^2) \left[ |\vm|^2-1 \right]\,dx + (K-C(1+t))\mathcal{H}^2(S_\vm),
\end{equation}
where $\vv:=\vm-\vn$. If we define the set 
\begin{equation}\label{5.1.14}
 \Gamma(S_\vm) := \left\{ \, (x,y) \, |\,\ (x,y,z) \in S_\vm\,\, \text{for some}\,\, z\,\right\},
\end{equation}
we can estimate the integral in \eqref{5.1.13} in the following way. By performing an rearrangement first devised in \cite{sethna1983relieving}, we obtain that 
\begin{equation}\label{5.1.14b}
 |\nabla \vv|+2t\vv \nabla \times \vv +t^2 |\vv|^2 = \sum_{i,j}\left( \vv_{i,j} - t\sum_k\epsilon_{ijk}\vv_k\right)^2 - t^2|\vv|^2 \geqslant -t^2|\vv|^2.
\end{equation}
We also know that $|\vv|\leqslant 2$ and $|\lambda|\leqslant \frac{\pi^2}{4} + t^2$. Hence the integrand in \eqref{5.1.13} is bounded below by $-C(1+t^2)$, therefore
\begin{equation}\label{5.1.15}
\begin{split}
 I(\vm)-I(\vn) \geqslant & \int_{\Gamma(S_\vn)^c} \int_0^d |\nabla \vv|^2+2t\vv\cdot\nabla \times \vv -\lambda |\vv|^2 \,dz \,dx\,dy \\ & -C(1+t^2) \mathcal{H}^2\left( \Gamma(S_\vm) \right) +(K-C(1+t))\mathcal{H}^2(S_\vm).
 \end{split}
\end{equation}
By the definition of $\Gamma(S_\vm)$ we know that for $\mathcal{H}^2$ almost every $(x,y) \in \Gamma(S_\vm)^c$, we have $\vv(x,y,\cdot) \in W^{1,2}_0\left( (0,1),\mathbb{R}^3 \right)$. Fortunately we can estimate the integrand above. It can be shown that $\lambda = -t^2 + \theta'(0)^2+2t^2 \sin^2  \theta$ and $\theta'(0)^2 \geqslant \frac{\pi^2}{4}$. Hence using Cauchy's inequality we find
\begin{equation}\label{5.1.15b}
\begin{split}
 &|\nabla \vv|^2 +2t \vv \cdot \nabla \times \vv - \lambda |\vv|^2 \\
 \geqslant & |\nabla \vv|^2 -2t|\vv|\,|\nabla \times \vv| - \left( \frac{\pi^2}{4} +t^2 \right) |\vv|^2 \\
 \geqslant & |\nabla \vv|^2 -2\sqrt{2}|\vv|\,|\nabla \vv| - \left( \frac{\pi^2}{4} +t^2 \right) |\vv|^2 \\
 \geqslant & |\nabla \vv|^2\left( 1- \frac{1}{2\sqrt{2}} \right)-\left( \frac{\pi^2}{4} +t^2+4\sqrt{2}t^2\right) |\vv|^2.
 \end{split}
\end{equation}
Therefore by using a Poincar\'{e} inequality with the specific constant $\pi^2$ \cite{dym1972fourier} we see that there exists some $t^*>0$, such that if $t \in (0,t^*)$ then 
\begin{equation}\label{5.1.16}
 I(\vm) - I(\vn) \geqslant \alpha \int_{\Gamma(S_\vm)^c}\int_0^d |\vm-\vn|^2 \,dz\,dx\,dy -C(1+t^2) \mathcal{H}^2\left( \Gamma(S_\vm) \right) +(K-C(1+t))\mathcal{H}^2(S_\vm).
\end{equation}
Finally we note that $\mathcal{H}^2 (\Gamma(S_\vm)) \leqslant \mathcal{H}^2 (S_\vm)$ so that if $K$ is sufficiently large then 
\begin{equation}\label{5.1.17}
 I(\vm)-I(\vn) \geqslant \alpha \int_{\Gamma(S_\vm)^c}\int_0^d |\vm-\vn|^2 \,dz\,dx\,dy  + \beta \mathcal{H}^2 (S_\vm) \geqslant 0.
\end{equation}
To demonstrate uniqueness we note that for equality to hold in \eqref{5.1.17}, we require that $\mathcal{H}^2 (S_\vm)=0$ so that $\vm \in W^{1,2}\left(\Omega ,\mathbb{S}^2 \right)$. But we already know that if $t<t^*$ then $\vn$ is the unique global minimum amongst such functions, hence it is the unique global minimum here as well.
 
\end{proof}

\begin{proposition}\label{4: cholesteric multidimensional minimisers}
 There exists some $\overline{t}>0$ such that if $t>\overline{t}$, there exists a state $\vn\in\mathcal{A}$ such that 
 \begin{equation}\label{5.77}
  I(\vn)<0,
 \end{equation}
whereas any $\vm = \vm(z) \in \mathcal{A}$ satisfies $I(\vm)\geqslant0$.

\end{proposition}

\begin{proof}

The state that we construct will not satisfy the boundary conditions anywhere on the horizontal faces. This will result in an energy penalty of $8K L_1L_2$. In the bulk we will piece together a periodic array of double twist cylinders, all lying perpendicular to the $(x,z)$-plane, in a hexagonal close packing system. This method is the most efficient way of close packing circles and would cover just over 90\% of the plane if the pattern was repeated indefinitely. The double twist cylinders that we use in this construction have a radius of $\frac{\pi}{2t}$.

\vspace{3mm}

\begin{figure}[ht]
\centering
\includegraphics[scale=0.3]{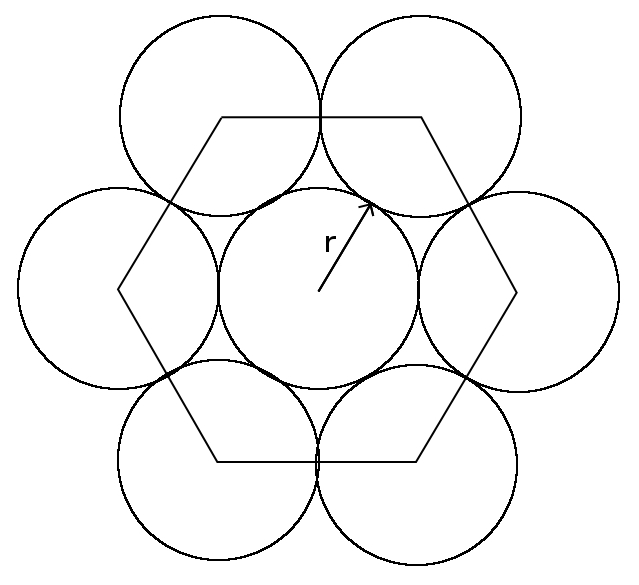}
\caption{Hexagonal Packing: Volume fraction occupied $= \frac{3\pi r^2}{6\sqrt{3} r^2} = 0.907$}
\label{packing}
\end{figure}

Importantly with the domain having width $2L_1$ and height $d$, the number of cylinders used in the pattern, $N_{cyl}$ has the following upper and lower bounds:
\begin{equation}\label{5.80}
 \lfloor 4L_1t \pi^{-1} \rfloor\, \lfloor dt (2\pi)^{-1} \rfloor \leqslant N_{cyl}\leqslant \lfloor 8L_1 d t^2 \pi^{-3} \rfloor.
\end{equation}
Where $\lfloor x\rfloor$ denotes the integer part of a real number. In other words, if $t$ is sufficiently large then there exist constants $A_1$, $A_2$, $A_3$, independent of $t$, such that 
\begin{equation}\label{5.81}
 A_1 t^2 - A_2 t \leqslant N_{cyl} \leqslant A_3 t^2.
\end{equation}
Around each of the cylinders we assign a discontinuity and in between the cylinders we simply set $\vn$ to be zero. Now we estimate the energy of this state by looking at the elastic and jump parts separately. With the upper bound on the number of cylinders we know that 
\begin{equation}\label{5.82}
 \mathcal{H}^2(S_{\vn})\leqslant 8L_1L_2+ (2L_2)\left(2\pi \left(\frac{\pi} {2t}\right)\right)A_3t^2\leqslant A+Bt,
\end{equation}
for some $A,B>0$. The integral of the elastic energy over each cylinder is $-\alpha <0$, a constant independent of $t$. Therefore using the lower bound for $N_{cyl}$ we get
\begin{equation}\label{5.83}
\begin{split}
\int_{\Omega}w(\vn,\nabla\vn)\,dx= &\int_{Cylinders}w(\vn,\nabla\vn)\,dx+\int_{\Omega\backslash Cylinders}w(\vn,\nabla\vn)\,dx\\
=&\int_{Cylinders}w(\vn,\nabla\vn)\,dx\\\leqslant&-\alpha A_1 t^2.\end{split}
\end{equation}
Hence
\begin{equation}\label{5.84}
 I(\vn)\leqslant -\alpha A_1 t^2+K(A+Bt).
\end{equation}
This means that if $t$ is sufficiently large then 
\begin{equation}\label{5.85}
 I(\vn)<0.
\end{equation}
Whereas any function of $\vm = \vm(z) \in \mathcal{A}$, satisfies $w(\vm,\nabla\vm)\geqslant0$ and hence $I(\vm)\geqslant0$.

\end{proof}

\begin{corollary}\label{4: n(z) not local min}
 There exists some $\overline{t}>0$ such that if $t>\overline{t}$ and $\vn = \vn(z) \in \mathcal{A}$ then $\vn$ is not a local minimiser of $I$.
\end{corollary}

\begin{proof}
 
 Take some $\vn =\vn(z) \in \mathcal{A}$. We remind ourselves that the domain in which our problem is set is the extended domain $\Omega_\epsilon$, therefore for some $\delta<<1$ we define
 \begin{equation}\label{5.86}
  A_\delta :=(-L_1+\delta,L_1-\delta)\times (-L_2+\delta,L_2 -\delta) \times \left(-\frac{\epsilon}{2},1+\frac{\epsilon}{2} \right),
 \end{equation}
 so that $A_\delta \subset \subset \Omega_\epsilon$. Take some $\vm \in \mathcal{A}$ such that $\vm = \vn$ on $\Omega \setminus \overline{A_\delta}$. Inside the set $A_\delta\cap \Omega$ we define $\vm$ to be the array of double twist cylinders as shown in Figure \ref{packing}. Then we know that the elastic energy of this configuration has an upper bound of $-\alpha \frac{A_1}{2} t^2$ (see \eqref{5.83}) if $\delta$ is sufficiently small. Hence
 \begin{equation}\label{5.87}
  \int_{A_{\frac{\delta}{2}}} w(\vm,\nabla \vm)\,dx + K\mathcal{H}^2 \left( S_\vm\cap A_{\frac{\delta}{2}} \right) \leqslant -\frac{\alpha A_1 t^2}{2} +  K\mathcal{H}^2 \left( S_\vm\cap A_{\frac{\delta}{2}} \right) \leqslant -\frac{\alpha A_1 t^2}{2}  + K(At+B),
 \end{equation}
 where the constants are independent of $\vn$ and $t$. Therefore if $t$ is sufficiently large we get 
 \begin{equation}\label{5.88}
   \int_{A_{\frac{\delta}{2}}} w(\vm,\nabla \vm)\,dx + K\mathcal{H}^2 \left( S_\vm\cap A_{\frac{\delta}{2}} \right) <  \int_{A_{\frac{\delta}{2}}} w(\vn,\nabla \vn)\,dx + K\mathcal{H}^2 \left( S_\vn\cap A_{\frac{\delta}{2}} \right)
 \end{equation}
 and $\left\{ \vn \neq \vm \right\} \subset\subset \Omega_\epsilon$. Hence $\vn$ is not a local minimiser of $I$.

\end{proof}

These final results are particularly interesting because they finally prove the necessary existence of multidimensional cholesteric patterns. We know that a minimiser of the problem must exist and it cannot be any function of one variable. Furthermore any function of one variable cannot even be a local minimiser. We did not include periodicity in $x$ and $y$ for our admissible functions when studying the nematic problem because it yielded a stronger result. If a function of $z$ is a global minimum without periodicity then it certainly will be with it. 

\vspace{3 mm}

The results in this paper have scratched the surface of what is possible to deduce about liquid crystals using functions of bounded variation. We showed in Section \ref{sec:2,3D results} that using a director theory, director fields of bounded variation are necessary to give finite energy to the defects. In the final sections it was demonstrated that by using this proposed model, the predictions obtained were consistent with the Sobolev space analysis in relatively unfrustrated systems and new in more frustrated settings. However further work is needed in this area to fully determine the advantages and disadvantages of the new model that we have proposed here.

\subsection*{Acknowledgements}

This work was supported by the EPSRC Science and Innovation
award to the Oxford Centre for Nonlinear PDE (EP/E035027/1). The author is supported by CASE studentship with Hewlett-Packard Limited.

\bibliographystyle{plain}
\bibliography{refs} 

\end{document}